\documentclass[12pt]{amsart}

\usepackage[utf8]{inputenc}
\usepackage{amsfonts}
\usepackage{amsthm}
\usepackage{amssymb}
\usepackage{amsmath}
\usepackage{amscd}
\usepackage{latexsym,dsfont}
\usepackage{bbm}
\usepackage{mathrsfs}

\usepackage{times}
\usepackage{microtype}
\usepackage{setspace}
\usepackage[margin=1.2in]{geometry}

\usepackage{cite}

\usepackage[colorlinks=true, pdfstartview=FitV, linkcolor=blue,
citecolor=blue, urlcolor=blue]{hyperref}

\newtheorem{theorem}{Theorem}[section]
\newtheorem{lemma}[theorem]{Lemma}
\newtheorem{prop}[theorem]{Proposition}
\newtheorem{cor}[theorem]{Corollary}

\newtheorem{defn}[theorem]{Definition}

\newtheorem{hyp}[theorem]{Hypothesis}

\theoremstyle{definition}

\theoremstyle{remark}
\newtheorem{remark}[theorem]{Remark}

\numberwithin{equation}{section}
 \allowdisplaybreaks

\def\Xint#1{\mathchoice
   {\XXint\displaystyle\textstyle{#1}}%
   {\XXint\textstyle\scriptstyle{#1}}%
   {\XXint\scriptstyle\scriptscriptstyle{#1}}%
   {\XXint\scriptscriptstyle\scriptscriptstyle{#1}}%
   \!\int}
\def\XXint#1#2#3{{\setbox0=\hbox{$#1{#2#3}{\int}$}
     \vcenter{\hbox{$#2#3$}}\kern-.5\wd0}}

\def\avgint{\Xint-}

\DeclareMathOperator{\Div}{div}
\DeclareMathOperator{\supp}{supp}

\newcommand{\op}{{\mathrm{op}}}
\DeclareMathOperator{\diag}{diag}
\DeclareMathOperator{\lip}{\mathrm{Lip}}

\newcommand{\N}{\mathbb N}

\newcommand{\R}{\mathbb{R}}
\newcommand{\rn}{{\mathbb{R}^n}}

\newcommand{\loc}{\mathrm{loc}}
\newcommand{\Ss}{\mathcal S}

\newcommand{\bH}{\mathbf H}
\newcommand{\bG}{\mathbf G}
\newcommand{\bR}{\mathbf R}
\newcommand{\bS}{\mathbf S}
\newcommand{\bT}{\mathbf T}

\newcommand{\cP}{\mathcal P}

\newcommand{\vecb}{\mathbf b}
\newcommand{\vecc}{\mathbf c}
\newcommand{\vece}{\mathbf e}
\newcommand{\vecf}{\mathbf f}
\newcommand{\vecg}{\mathbf g}
\newcommand{\vecr}{\mathbf r}
\newcommand{\vecs}{\mathbf s}
\newcommand{\vecv}{\mathbf v}

\newcommand{\vech}{\mathbf h}
\newcommand{\vecu}{\mathbf u}
\newcommand{\vecw}{\mathbf w}

\title[Existence and uniqueness of solutions]
{Existence and uniqueness of solutions of degenerate elliptic equations with lower order terms }

\author[\c{C}etin, {\em et al.}] {  \c{S}eyma \c{C}etin, David Cruz-Uribe OFS,  Feyza Elif Dal,\\ Scott Rodney,  and Yusuf Zeren}

\address{\c{S}eyma \c{C}etİn \\
Dept. of Software Engineering, Istanbul Gelisim University, Avcilar, TUrkey
}
\email{seycetin@gelisim.edu.tr}

\address{David Cruz-Uribe, OFS \\
Dept. of Mathematics \\
University of Alabama \\
 Tuscaloosa, AL 35487, USA}

\email{dcruzuribe@ua.edu}

\address{Feyza Elif Dal \\
Dept. of Mathematics \\
Y\i ld\i z Technical University \\
Esenler, Istanbul, 34220 Davutpasha, Turkey
}
\email{feyzadal@hotmail.com}

\address{Scott Rodney\\
Dept. of Mathematics, Physics and Geology \\ 
Cape Breton University \\
Sydney, NS B1Y3V3, CA} 

\email{scott\_rodney@cbu.ca}

\address{Yusuf Zeren \\
Dept. of Mathematics \\
Y\i ld\i z Technical University \\
Esenler, Istanbul, 34220 Davutpasha, Turkey
}

\email{yzeren@yildiz.edu.tr}
\thanks{The second author is partially supported by a Simons Foundation
  Travel Support for Mathematicians Grant and by NSF Grant DMS-2349550.
  The third author is supported by the TUBITAK 2211-E Domestic Direct Doctorate Scholarship Program. The fourth author is partially supported by an NSERC development grant.  This project is supported by  TUBITAK, the Scientific and Technological Research Council of T\"urkiye through a 2501 Joint Research Program grant 223N112; the fifth author is the PI on this grant.  Finally, the authors would like to thank the anonymous referee for their detailed and thoughtful report, which helped to significantly improve the paper.}

\keywords{degenerate elliptic equations, degenerate Sobolev inequalities, existence and uniqueness}
\subjclass{35A01, 35A02, 35H20, 35J15, 35R05, 46E35}

\makeatletter
\def\l@subsection{\@tocline{2}{0pt}{4pc}{5pc}{}}
\makeatother

\begin{document}
\begin{abstract}
 We prove the existence and uniqueness of solutions to a Dirichlet problem 
 \[ \begin{cases}
Lu = f + v^{-1}\Div(v\vece h), & x \in \Omega, \\
 u = 0, & x \in \partial \Omega,
 \end{cases}\]
 where $L$ is a degenerate, linear, second order elliptic operator with lower order terms.  We assume very weak hypotheses, in terms of the coefficients of the equation, and we also assume the existence of degenerate Sobolev and Poincar\'e inequalities.  One notable feature of our result is that we show that we can assume significantly weaker versions of the Sobolev inequality if we in turn assume stronger integrability conditions on the coefficients.  Our theorems generalize a number of results in the literature on degenerate elliptic equations.  
\end{abstract}

\maketitle


\section{Introduction}
\label{sec:introduction}

The goal of this paper is to prove the existence and uniqueness of degenerate weak solutions to  Dirichlet problems such as
\begin{equation} \label{eqn:intro-dirichlet}
\begin{cases}
Lu = f + v^{-1}\Div(v \vece h), & x \in \Omega, \\
 u = 0, & x \in \partial \Omega,
 \end{cases}
\end{equation}
where $\Omega\subset \R^n$ is a bounded domain and $L$ is a degenerate, second order elliptic operator with lower order terms, e.g.,
\begin{equation} \label{eqn:intro-operator}
 Lu = v^{-1} \Div(Q\nabla u) + \vecb \cdot \nabla u + v^{-1}\Div(v\vecc u) + du.  
 \end{equation}
Here, $Q$ is an $n\times n$ self-adjoint, positive semidefinite, measurable matrix function, $v$ is a nonnegative measurable function (``weight"), and the  coefficients $\vecb$, $\vecc$, $d$, and the data $\vece$, $h$,  are measurable functions.  
(In fact, we will consider a more general equation in our main results.  See Section~\ref{section:prelim} below.)  Our goal is to determine the weakest assumptions on the weight $v$, the matrix $Q$, and the coefficients to show that a solution exists and is unique.  

To put our results into context, we describe some previous results.  Our goal is not to be comprehensive but to indicate some prior results that motivated our work.  
In the classical case, when $v=1$ and $Q$ is uniformly elliptic, the existence of solutions to this Dirichlet problem was studied by multiple authors and there is a complete theory:  we refer the readers to Gilbarg and Trudinger~\cite{MR1814364} and Ladyzhenskaya and Ural$'$tseva~\cite{MR244627} for details and extensive references. Here we note that the classical Sobolev and Poincar\'e inequalities play a central role in the proofs of these results.  

In the degenerate case, Fabes, Kenig and Serapioni~\cite{MR643158} studied the case when $Q$ satisfies the ellipticity condition
\[ \lambda v(x) |\xi|^2 \leq \langle Q\xi, \xi \rangle \leq \Lambda v(x)|\xi|^2\]
for $\xi \in \R^n$, where the weight $v$ satisfies the Muckenhoupt $A_2$ condition.  They extended much of the classical theory proving, existence, uniqueness and regularity for the principal part of the operator, that is, $Lu= v^{-1} \Div(Q\nabla u)$.  (See also Stredulinksy~\cite{MR757718}.)   Key to their work was proving the existence of a global degenerate Sobolev inequality 
\[ \bigg( \int_\Omega |\varphi|^{2\sigma}\,vdx\bigg)^{\frac{1}{2\sigma}}
\leq C\bigg( \int_\Omega |\nabla \varphi|^{2}\,vdx\bigg)^{\frac{1}{2}}, \]
where $1< \sigma < \frac{n}{n-1}+\delta$.  (Note that, in general, the ``gain" $\sigma$ is less than in the classical Sobolev inequality, where $\sigma=\frac{n}{n-2}$.)   Chanillo and Wheeden~\cite{MR0847996} proved two weight global and local Sobolev and Poincar\'e inequalities of the form
\[ \bigg( \int_\Omega |\varphi|^{2\sigma}\,vdx\bigg)^{\frac{1}{2\sigma}}
\leq C\bigg( \int_\Omega |\nabla \varphi|^{2}\,wdx\bigg)^{\frac{1}{2}}, \]
where the weight $v$ is doubling, $w$ satisfies the $A_2$ condition, and the pair satisfies a certain balance condition involving $\sigma$.  They used these to study the regularity of solutions of the Dirichlet problem (though not existence), again for the principal part only, when $Q$ satisfies
\begin{equation} \label{eqn:intro-elliptic}  
w(x) |\xi|^2 \leq \langle Q\xi, \xi \rangle \leq  v(x)|\xi|^2.
\end{equation}
(See also Franchi, Lu, and Wheeden~\cite{MR1354890}.)

In these results the authors worked in a specific ``geometric" setting (determined by the matrix $Q$ and the weights $v$ and $w$).  A different,``abstract" approach was introduced by Guti\'errez and Lanconelli~\cite{MR2015404}, and expanded upon by Sawyer and Wheeden~\cite{MR2574880,MR2204824}.  These authors considered a version of our Dirichlet problem, assuming $v=1$.  Their approach was to establish a collection of assumptions on $Q$ and the coefficients that were sufficient to prove their results.  Chief among them was assuming the existence of a global Sobolev inequality of the form
\begin{equation} \label{eqn:intro-sobolev}
 \bigg( \int_\Omega |\varphi|^{2\sigma}\,vdx\bigg)^{\frac{1}{2\sigma}}
\leq C\bigg( \int_\Omega |\sqrt{Q} \nabla \varphi|^{2}\,dx\bigg)^{\frac{1}{2}}, 
\end{equation}
where $\sigma>1$.
This approach was adopted by the fourth author and his collaborators in a series of papers~\cite{MR2994671,MR2551502,MR2711279,MR3388872,MR3369270,MR3095112,MR2906551,MR4069009}.

The existence of the degenerate Sobolev inequality~\eqref{eqn:intro-sobolev} is closely related to the property that metric balls in the Carnot-Caratheodory metric are doubling.  (See Korobenko, Maldonado and Rios~\cite{MR3359590}.)  In the infinitely degenerate setting, e.g., when the matrix $Q$ has the form
\[ Q(x) = \begin{pmatrix} 1 & 0 \\
0 & \exp(-|x|^{-2}) \end{pmatrix}, \]
this metric is no longer doubling.  This led to Korobenko, Rios, Sawyer, and Shen~\cite{MR4224718,KRSS2024} to consider the abstract approach but with a Sobolev inequality with the gain in the scale of Orlicz spaces.  That is, they assumed Sobolev inequalities of the form
\[ \|\varphi\|_{L^A(\Omega)} \leq C \int_\Omega |\sqrt{Q} \nabla \varphi|\,dx, \]
where $L^A$ is the Orlicz space induced by a Young function of the form $A(t)=t\log(e+t)^\sigma$, $\sigma>1$. They used this inequality to develop De Giorgi and Moser iteration and prove regularity of weak solutions.   
(To be precise, in~\cite{KRSS2024} they assumed a stronger, ``modular" inequality required to prove Moser iteration.  This distinction does not concern us and we refer the reader to this paper for more information.) They also proved ``geometric" results by giving specific examples where their Sobolev inequality holds.
More recently, Hafeez, Lavier,  Williams,  and Korobenko~\cite{MR4322905} proved the existence and uniqueness of solutions to the Dirichlet problem for the principal part, assuming an Orlicz-Sobolev inequality of the form
\begin{equation} \label{eqn:intro-Orlicz-Sobolev}
 \|\varphi\|_{L^A(\Omega)} \leq C \bigg(\int_\Omega |\sqrt{Q} \nabla \varphi|^2\,dx\bigg)^{\frac{1}{2}}, 
 \end{equation}
where $A(t)=t^2\log(e+t)^\sigma$, $\sigma>0$.  (Korobenko~\cite{MR4277804} and later Heikkinen and Karak~\cite{MR4331593} considered the connection between doubling and Orlicz-Sobolev inequalities.)

The fully degenerate equation, with $v\neq 1$ and $Q$ and $v$ satisfying the second inequality in~\eqref{eqn:intro-elliptic}, has been considered by the second and fourth authors and their collaborators, though only for the principal part~\cite{MR4280269,MR3011287,MR3846744, MR2771262,CUMR2025,CUMR2025,CUMS2025}. In most of these papers they assumed Sobolev inequalities with gain in the scale of Lebesgue spaces, but in~\cite{CUMR2025} they assumed an Orlicz-Sobolev inequality of the form~\eqref{eqn:intro-Orlicz-Sobolev}. They also proved (weak) regularity results assuming only a Sobolev inequality with no gain (that is, \eqref{eqn:intro-sobolev} with $\sigma=1$).  Earlier, in~\cite{MR3846744}, they proved the existence of solutions to a Neumann-type problem for a degenerate $p$-Laplacian, assuming only a Poincar\'e inequality with no gain.  

When $v\neq 1$, we want to highlight one feature of our equation: the weight $v^{-1}$ on the principal term, and the conjugation $v^{-1}\Div( v\cdot )$ that appears in one of the first order terms and in the data.  Normalizing by dividing by the weight $v$ is implicit in Fabes, Kenig and Serapioni~\cite{MR643158}, and also in Chiarenza and Serapioni~\cite{MR799906}, and was made explicit in~\cite{MR3335399,MR2901220,MR2419963}.  The intuition for this is that the weak derivatives in the divergence are formally defined by pairings in the Hilbert space $L^2(v,\Omega)$, but we also want them to agree with the results obtained by (formally) integrating by parts.  In other words, if $\varphi$ is in $\lip_0(\Omega)$, we want 
\[ \int_\Omega v^{-1}\Div(v \vecu) \varphi\, vdx = \int_\Omega \Div(v \vecu) \varphi \,dx = -\int_\Omega \vecu \cdot \nabla \varphi\, vdx.  \]
This definition of weak derivatives is described in more detail in Section~\ref{section:prelim} below.
\medskip

To illustrate our results, we state two special cases.  The first is for the principal part of the equation only and is restated below as Theorem~\ref{thm:no-gain-no-first-order}.

\begin{theorem} \label{thm:intro-principal}
Given an open, bounded, connected set $\Omega \subset \R^n$, a nonnegative, measurable function $v$ on $\Omega$,  and a real, self-adjoint, positive semidefinite, measurable matrix function $Q$ that is finite almost everywhere on $\Omega$, suppose the Sobolev inequality without gain
(that is, \eqref{eqn:intro-sobolev} with $\sigma=1$)
holds
for all $\varphi \in \lip_0(\Omega)$.  If $f,\, h\in L^2(v,\Omega)$ and $\vece$ is a degenerate subunit vector field, that is, for all $\xi \in \R^n$, $\vece$ satisfies
        \[ \langle \vece, \xi \rangle^2 \leq v^{-1} \langle Q\xi, \xi \rangle, \]
then the Dirichlet problem
\begin{equation*}
\begin{cases}
-v^{-1}\Div(Q\nabla u) = f + v^{-1}\Div(v\vece h), & x \in \Omega, \\
 u = 0, & x \in \partial \Omega,
 \end{cases}
\end{equation*}
has a unique degenerate weak solution such that
\[  \int_\Omega |u|^{2}\,vdx + \int_\Omega |\sqrt{Q} \nabla u|^{2}\,dx < \infty.  \]
\end{theorem}

\begin{remark}
    The definition of a degenerate subunit vector field is a generalization of the subunit vector fields defined when $v=1$ by Fefferman and Phong~\cite{MR0730094} and later used by Sawyer and Wheeden~\cite{MR2204824,MR2574880}.  Below, we will show that we can replace this pointwise inequality with a weaker norm inequality:  see~\eqref{eqn:norm-subunit}.
\end{remark}

\begin{remark}
    Our data term $v^{-1}\Div(v\vece h)$ appears different than the divergence term $\Div(\vecg)$, $\vecg \in L^2(\Omega,\R^n)$, that appears in the classical equation.   However, if $v=1$ and $Q$ is uniformly elliptic, and if we assume after renormalizing that $|Q|_\op \leq 1$, then the subuniticity condition reduces to 
$\langle \vece, \xi \rangle^2 \leq |\xi|^2$,  $\xi \in \R^n$. 
    It follows from this that $|\vece|\leq 1$ a.e.  Therefore, given $\vecg \in L^2(\Omega)$ we can rewrite it as
    \[ \vecg = \frac{\vecg}{|\vecg|} |\vecg| = \vece h.\]
\end{remark}

\begin{remark}
    Theorem~\ref{thm:intro-principal} was proved by Hafeez {\em et al.}~\cite[Theorem~1.1]{MR4322905} with $\vece,\, h=0$, and assuming that $v=1$.  In the statement of their result they assume that $f, \,|Q|_\op \in L^\infty(\Omega)$ and they assume an Orlicz-Sobolev inequality, but it is clear from the proof that they only need $f\in L^2(\Omega)$ and a Sobolev inequality with no gain.  Guti\'errez and Lanconelli~\cite{MR2015404} considered the  Dirichlet problem with nonzero boundary values
    \begin{equation*}
\begin{cases}
\Div(Q\nabla u) = 0 & x \in \Omega, \\
 u = \varphi, & x \in \partial \Omega.
 \end{cases}
\end{equation*}
They also assume that $|Q|_\op \in L^\infty(\Omega)$ and a Sobolev inequality with gain, but again from their proof they only need a Sobolev with no gain.  Their result is a special case of Corollary~\ref{cor:boundary-data} below.
\end{remark}

Our second result is a special case of Theorem~\ref{thm:main-theorem} below.  The hypotheses are considerably more restrictive than those for Theorem~\ref{thm:intro-principal}; the bulk of them, and in particular the Poincar\'e inequality, are required to estimate the first order terms.  For brevity, we omit some minor technical hypotheses.

\begin{theorem} \label{thm:intro-main}
    Given an open, bounded, connected set $\Omega \subset \R^n$, a nonnegative function $v\in L^1(\Omega)$,  and a real, self-adjoint, positive semidefinite, measurable matrix function $Q$ that satisfies $|Q(x)|_\op \leq k v(x)$ $v$-a.e., suppose further that:
    \begin{enumerate}
        \item The Sobolev inequality~\eqref{eqn:intro-sobolev} holds for some $\sigma>1$ and all $\varphi\in \lip_0(\Omega)$;
        \item A weak Poincar\'e inequality holds:  for every $\epsilon>0$ there exists $\delta>0$ such that if $0<r<\delta$,
        \[ \int_{B(x_0,r)}  |\varphi-\varphi_{B,v}|^2 \,vdx 
        \leq \epsilon\int_{B(cx_0,r)} |\sqrt{Q}\varphi|^2 \,dx,  \]
        where $c\geq 1$, $\varphi \in \lip_0(\Omega)$,  $B(x_0,cr)\subset \Omega$ and $\varphi_{B,v} =v(B)^{-1}\int_B \varphi\,vdx$;
        
        \item $f,\, h \in L^2(v,\Omega)$;

        \item For some $q>2\sigma'$, where $\sigma'=\frac{\sigma}{\sigma-1}$, $d \in L^{\frac{q}{2}}(v,\Omega)$;
        
        \item $\vecb,\, \vecc,\, \vece$ are degenerate subunit vector fields.
    \end{enumerate}
    Then the Dirichlet problem~\eqref{eqn:intro-dirichlet} with $L$ defined by~\eqref{eqn:intro-operator} has a unique solution.  
\end{theorem}

\begin{remark}
   When $v=1$, Theorem~\ref{thm:intro-main} was proved by the fourth author~\cite{MR2994671}.   In the weighted case (i.e., $v\neq 1$) it is new, even in the ``one weight" case considered by Fabes, Kenig, and Serapioni~\cite{MR643158}.  We emphasize that with the exception of the Sobolev and Poincar\'e inequalities, we only assume integrability conditions on $v$, $Q$, and the coefficients.  Further, we have recently shown, using techniques derived from the study of Rubio de Francia extrapolation in harmonic analysis, that  the Sobolev and Poincar\'e  inequalities hold with only integrability assumptions on $v$, $w$, and $Q$:  see~\cite{DCU-FED-SR, DCU-FED-SR-2}.  In particular, we do not have to assume that $Q$ is continuous or that $v$ and $w$ are doubling or Muckenhoupt $A_2$ weights.
\end{remark}

\begin{remark}
   Theorem~\ref{thm:intro-main} is an ``abstract" result but it can be immediately used to prove ``geometric" results once degenerate Sobolev and Poincar\'e inequalities are established.  Towards this end, we refer the reader to the specific geometries and (corresponding inequalities) proved by~Fabes, Kenig, and Serapioni~\cite{MR643158}, Chanillo and Wheeden~\cite{MR0847996}, and Korobenko, {\em et al.}~\cite{MR4224718}.  See also our recent paper~\cite{DCU-FED-SR-2}.
\end{remark}

\begin{remark}
    Though we do not state it here explicitly, we also prove versions of Theorem~\ref{thm:intro-main} only assuming a Sobolev inequality in the scale of Orlicz spaces~\eqref{eqn:intro-Orlicz-Sobolev} or assuming a Sobolev inequality with no gain (i.e., $\sigma=1$) in~\eqref{eqn:intro-sobolev}. See Theorems~\ref{thm:log-gain} and~\ref{thm:no-gain} below.   To the best of our knowledge, these results on existence and uniqueness for equations with lower order terms are new.  In~\cite{CUMR2025}, the authors considered the boundedness of solutions of equations with only the principal term on the left-had side assuming the existence of such Sobolev inequalities; our results here show that those results are not vacuous.
\end{remark}

\begin{remark}
    One difference between Theorem~\ref{thm:intro-principal} and Theorem~\ref{thm:intro-main} is that for the latter we need to assume $v\in L^1(\Omega)$ and $|Q(x)|_\op \leq kv(x)$ a.e., whereas for the equation that consists of the principal, second order term, we only need that $v$ is non-negative and measurable.  These additional assumptions are only required to handle the estimates for equations with first order terms.  Throughout we will highlight where we use these hypotheses on $v$ and $Q$.
\end{remark}

\medskip

The remainder of this paper is organized as follows.  In Section~\ref{section:prelim} we gather a number of definitions and results about function spaces and degenerate Sobolev and Poincar\'e inequalities.  The purpose of this is to allow us to give a precise definition of the elliptic equations we are working with, define their degenerate weak solutions, and to  state the hypotheses of our main results.  We state these results in the greatest possible generality; in particular, for all $p$, $1\leq p<\infty$, and not simply for $p=2$.  We do so since we are drawing on results from several different sources with slightly different hypotheses, and we thought it was useful to establish a consistent framework, both for this paper and for future work.  

Readers who are willing to take some things on faith might want to skip directly to Section~\ref{section:main-results}, where we state our existence and uniqueness theorems.  We actually give four different results, each of which depends on the Sobolev and Poincar\'e inequalities that we assume:  roughly, the weaker the inequalities we assume, the stronger the integrability requirements needed on the coefficients.   As an immediate corollary, we state an existence and uniqueness theorem for Dirichlet problems with nontrivial boundary data.  

The subsequent sections give the proofs of our main results.  In Section~\ref{section:prelim-vector} we prove a number of lemmas about the vector fields used to define the first order terms in our equation. We also establish that with our hypotheses we can use a larger collection of functions as test functions in our definition of a degenerate weak solution.   In Section~\ref{section:zero} we prove that with our hypotheses, the zero solution is the unique solution to the homogeneous problem.  In Sections~\ref{section:coerce-power-gain} and~\ref{section:coerce-weaker-gain} we show that the bilinear form induced by our elliptic equation is bounded and almost coercive.  In Sections~\ref{section:zero} to~\ref{section:coerce-weaker-gain} we must give multiple proofs of results depending on our assumed Sobolev inequalities, but we have tried to eliminate as much repetition as possible.  Finally, in Sections~\ref{section:main-proofs} and~\ref{section:cor-proofs} we use the functional analytic techniques--that is, the Lax-Milgram theorem and the Fredholm alternative--to prove the existence and uniqueness of degenerate weak solutions to our Dirichlet problem.  

\section{Preliminary definitions: function spaces}
\label{section:prelim}

  Throughout this paper, the constant $n$ will always denote the dimension of the underlying Euclidean space, $\R^n$.   We will denote constants in computations by $C$ and $c$; these values may change at each appearance.  If they depend on underlying parameters,  we will write, for example, $C(n)$.   Given two quantities $A$ and $B$, if there exists $c>0$ such that $A\leq cB$, we will write $A\lesssim B$.  If $A\lesssim B$ and $B\lesssim A$, we will write $A\approx B$.

The set $\Omega$ will always be assumed to be a domain:  that is, an open, connected, bounded set in $\rn$.
A weight will be a nonnegative measurable function $v$ defined on $\Omega$ such that $0<v(x)<\infty$ a.e.  (Note that we do not assume {\em a priori} that $v$ is locally integrable.)
Given a set $E\subset \Omega$ and weight $v$, we define the "$v$-measure" of $E$ by
\[ v(E) = \int_E \,vdx.  \]
If  $0<|E|<\infty$, the average of a function $f$ on $E$
is
\[ f_E = \avgint_E f\,dx = \frac{1}{|E|} \int_E f\,dx.
\]
If $0<v(E)<\infty$, the weighted average of $f$ with respect to a weight $v$ is
\begin{equation} \label{eqn:wtd-avg}
f_{v,E} = \avgint_E f \,dv = \frac{1}{v(E)}\int_E f\,\,vdx. 
\end{equation}
We will use Young's inequality repeatedly in the following form:  given $1<p<\infty$ and $\eta>0$, if $a,\,b\geq 0$, then 
\begin{equation} \label{eqn:youngs-ineq}
 ab \leq \eta a^p+ \eta^{-\frac{p'}{p}}b^{p'}. 
 \end{equation}

\subsection*{Vectors and matrices}
Vectors in $\R^n$  are assumed to be column vectors, but we will on occasion abuse notation and write $\vecv=(v_1,\ldots,v_n)$.    The standard basis in $\R^n$ will be denoted by $\{\vece_j\}_{j=1}^n$.  
Given an $n\times n$ matrix $A$, we define the operator norm of $A$ by
\[ |A|_\op = \sup\{ |A\xi| : \xi \in \R^n, |\xi|=1 \} . \]
Let $\Ss_n$ denote the collection of real, self-adjoint, $n\times n$ matrices that are positive semidefinite.  

Hereafter, $Q : \Omega \rightarrow \Ss_n$ will be a matrix function whose entries are Lebesgue measurable scalar functions that are finite almost everywhere.   We will also assume  that $Q(x)$ is invertible for a.e.~$x\in \Omega$.  The matrix $Q$ is diagonalizable:
\[ Q(x) = U^t(x) D(x) U(x),\]
where $U$ is orthogonal and
\[ D(x) = \diag(\lambda_1(x),\ldots,\lambda_n(x)) \]
is a diagonal matrix.  Since $Q$ is measurable, $U$ and the eigenvalues $\lambda_i$ can be chosen to be measurable as well.  (See~\cite[Lemma~2.3.5]{MR1350650}.) By assumption $0<\lambda_i(x)<\infty$ for all $i$ and a.e.~$x\in \Omega$.  Define powers $Q^a$, $a>0$, by 
\[ Q^a(x) = U^t(x) D(x)^a U(x),\]
where $D^a = \diag(\lambda_1^a,\ldots,\lambda_n^a)$.  As a consequence of this diagonalization, we have that  operator norm commutes with taking powers: $|Q^a|_\op =|Q|_\op^a$. We also have the identity $(\sqrt{Q})^{-1}=\sqrt{Q^{-1}}$.

Given $1\leq p<\infty$ and a matrix function $Q$, there exist weights $\tilde{v}$ and $\tilde{w}$ such that for a.e.~$x\in \Omega$ and every $\xi\in \R^n$,  
\[ \tilde{w}(x)|\xi|^p \leq |\sqrt{Q(x)}\xi |^p  \leq \tilde{v}(x)|\xi|^p.  \]
For instance, we could take  $\tilde{v}$ and $\tilde{w}$ to be the appropriate powers of the largest and smallest eigenvalues of $Q$:
\[ \tilde{v}(x) = |Q(x)|_\op^{\frac{p}{2}}, \qquad \tilde{w}(x) = |Q^{-1}(x)|_\op^{-\frac{p}{2}}. \]
(See~\cite[Proposition~3.2]{MR3544941}; note this result is stated in a different but equivalent form.)  However, in practice it  suffices, for instance, for $\tilde{v}$ to satisfy  $|Q(x)|_\op^{\frac{p}{2}} \leq \tilde{v}(x)$.  In fact, it is often enough to assume $|Q(x)|_\op^{\frac{p}{2}} \leq k\tilde{v}(x)$ for some $k>0$.  

\subsection*{Function spaces}
Given a weight $v$ and $1\leq p<\infty$, let $L^p(v,\Omega)$  be the measurable functions defined on $\Omega$ such that
\[ \|f\|_{L^p(v,\Omega)} = \bigg(\int_\Omega |f|^p \,\,vdx \bigg)^{\frac{1}{p}} < \infty. \]

We gather some well known results about $L^p(v,\Omega)$ in a lemma.  See~\cite[Theorem~3.11, Exercise~6.6, Theorem~3.14]{MR924157} for further details.

\begin{lemma} \label{lemma:lpv-banach}
Given $1\leq p<\infty$ and a weight $v$, the following are true:
\begin{enumerate}
    \item $L^p(v,\Omega)$ is a Banach space.

    \item If $1<p<\infty$, $L^p(v,\Omega)$ is reflexive and its dual space is (isomorphic to) $L^{p'}(v,\Omega)$, with the dual pairing 
    \[ \langle f, g \rangle_{L^p(v,\Omega)} = \int_\Omega fg \,vdx.  \]
    In particular, when $p=2$, $L^p(v,\Omega)$ is a Hilbert space.

    \item If $v\in L^1_\loc(\Omega)$, then $L^p(v,\Omega)$ is separable.
    
\end{enumerate}
    
\end{lemma}

Note that if $v\in L^1(\Omega)$, then constant functions are contained in $L^p(v,\Omega)$.  For vector-valued functions $\vecf : \Omega \rightarrow \R^d$, $d\in \N$, we will write $\vecf \in L^p(v,\Omega, \R^d)$ if
$|\vecf|\in L^p(v,\Omega)$.  For brevity we will write $\|\vecf\|_{L^p(v,\Omega)}$ instead of $\||\vecf|\|_{L^p(v,\Omega)}$ or $\|\vecf\|_{L^p(v,\Omega,\R^d)}$.

Given a matrix function $Q$, let $QL^p(\Omega)$  be the space of measurable functions $\vecf : \Omega \rightarrow \rn$ such that
\[ \|\vecf\|_{QL^p(\Omega)} = \bigg(\int_\Omega |\sqrt{Q}\vecf|^p\,dx\bigg)^{\frac{1}{p}}
=  \bigg(\int_\Omega \langle Q\vecf,\vecf \rangle^{\frac{p}{2}}\,dx\bigg)^{\frac{1}{p}} < \infty. \]
This space and its properties have been studied extensively; see \cite{CUMR2025} and \cite{MR4280269,MR3846744,MR3369270,MR3388872} where it is referred to as $\mathcal{L}^p_Q(\Omega)$.
Part of the following result was  proved in~\cite[Lemma~2.1]{MR3846744}; we repeat the short proof here to make clear the role played by the hypotheses.  (See also~\cite{MR4332462}.)

\begin{lemma} \label{lemma:Qspace}
Given $1\leq p < \infty$ and a matrix function $Q$, the following are true:
\begin{enumerate}
    \item $QL^p(\Omega)$ is a Banach space.

    \item If $1<p<\infty$, $QL^p(\Omega)$ is reflexive, and its  dual space  is (isomorphic to) $QL^{p'}(\Omega)$, with the dual pairing
    \[ \langle \vecf, \vecg \rangle_{QL^p(\Omega)} = \int_\Omega Q\vecf\cdot \vecg\,dx =
    \int_\Omega \sqrt{Q}\vecf \cdot \sqrt{Q}\vecg\, dx. \]
    In particular, if $p=2$, $QL^2(\Omega)$ is a Hilbert space.

    \item If  $|Q|_\op^{\frac{p}{2}} \in L^1_\loc(\Omega)$, then $QL^p(\Omega)$ is separable.
\end{enumerate}    
\end{lemma}

\begin{proof}
    To prove $(1)$, note that it is immediate that $QL^p(\Omega)$ is a normed vector space. Hence, we only have to prove it is complete.  Let $\{\lambda_i\}_{i=1}^n$ be the (nonnegative) eigenvalues of $Q$.  Since the associated eigenfunctions $\{\vecv_i\}_{i=1}^n$ are the columns of the diagonalizing orthogonal matrix $U$ described above, the $\vecv_i$ are  measurable functions and form an orthonormal basis.  Given any function $\vecf \in QL^p(\Omega)$, we can uniquely decompose it as
    \[ \vecf = \sum_{i=1}^n f_i \vecv_i,\]
    where the $f_i$ are measurable scalar functions.  
    Hence,
    \[ \sqrt{Q}\vecf =  \sum_{i=1}^n f_i \lambda_i^{\frac12} \vecv_i,  \]
    and so we have that
    \[ \bigg(\int_\Omega |\sqrt{Q}\vecf|^p\,dx\bigg)^{\frac1p} 
    \approx \sum_{i=1}^n \bigg(\int_\Omega |f_i|^p \lambda_i^{\frac{p}{2}}\,dx\bigg)^{\frac{1}{p}},  \]
    where the implicit constants depend only on $p$ and $n$.  Hence, $QL^p(\Omega)$ is isomorphic to the finite direct sum  $\bigoplus_{i=1}^n L^p(\lambda_i^{\frac{p}{2}},\Omega)$.    Since $\lambda_i^{\frac{p}{2}}$ is a weight, by Lemma~\ref{lemma:lpv-banach}, $L^p(\lambda_i^{\frac{p}{2}},\Omega)$ is complete. The finite direct sum of complete spaces is complete; thus $QL^p(\Omega)$ is a Banach space. 

    If $1<p<\infty$, then, by Lemma~\ref{lemma:lpv-banach}, each space $L^p(\lambda_i^{\frac{p}{2}},\Omega)$ is reflexive.  Since the finite direct sum of reflexive spaces is reflexive, $QL^p(\Omega)$ is reflexive. To identify the dual space, note that given any $\vech \in QL^{p'}(\Omega)$
    we have that 
    \[ \Lambda_\vech (\vecf) = \int \sqrt{Q}\vecf\cdot \sqrt{Q}\vech \,dx \]
    induces an element of $QL^p(\Omega)^*$. To show that all elements of the dual are of this form,  we consider the dual of the direct sum.  By Lemma~\ref{lemma:lpv-banach}, $L^p(\lambda_i^{\frac{p}{2}},\Omega)^*$ is (isomorphic to) $L^{p'}(\lambda_i^{\frac{p}{2}},\Omega)$, so every linear functional $\tilde{\Lambda}$ on the direct sum is given by 
    \[ \tilde{\Lambda}(f_1,\ldots, f_n) = \sum_{i=1}^n \int_\Omega f_i(x)g_i(x)\lambda_i(x)^{\frac{p}{2}}\,dx, \]
    where $g_i \in L^{p'}(\lambda_i^{\frac{p}{2}},\Omega)$.  Now fix $\Lambda \in QL^p(\Omega)^*$ and let $\tilde{\Lambda}$ be the element of the dual space of the direct sum that it is isomorphic to.  Define $h_i= \lambda_i^{\frac{p-2}{2}}g_i$ and let %
    \[ \vech(x)= \sum_{i=1}^n h_i\vecv_i. \]
    Then, arguing as we did above,
    \begin{multline*} \|\vech\|_{QL^{p'}(\Omega)} = \bigg(\int_\Omega |\sqrt{Q}\vech|^{p'}\,dx\bigg)^{\frac{1}{p'}}
    \approx \sum_{i=1}^n \bigg(\int_\Omega |h_i|^{p'} \lambda_i^{\frac{p'}{2}}\,dx \bigg)^{\frac{1}{p'}} \\
    = \sum_{i=1}^n \bigg(\int_\Omega |g_i|^{p'} \lambda_i^{\frac{p'}{2}-\frac{p'(p-2)}{2}}\,dx \bigg)^{\frac{1}{p'}}
    = \sum_{i=1}^n \bigg(\int_\Omega |g_i|^{p'} \lambda_i^{\frac{p}{2}}\,dx \bigg)^{\frac{1}{p'}} < \infty.  
    \end{multline*}
    Moreover, 
    \[ \Lambda_{\vech}(\vecf) 
     = \int_\Omega \sqrt{Q}\vecf \cdot \sqrt{Q}\vech\,dx = \sum_{i=1}^n f_ih_i \lambda_i 
     = \sum_{i=1}^n f_ig_i \lambda_i^{\frac{p-2}{2}+1} = \tilde{\Lambda}(f_1,\ldots, f_n).  \]
    Therefore, $\Lambda= \Lambda_\vech$.
we have that, up to isomorphism, $\Lambda=\Lambda_\vech$.  Thus $QL^p(\Omega)^*$ is isomorphic to $QL^{p'}(\Omega)$.    When $p=2$ it is immediate that this dual pairing is an inner product, so $QL^2(\Omega)$ is a Hilbert space.  
    
    Finally, if $|Q|_\op^{\frac{p}{2}} \in L^1_\loc(\Omega)$, then each function $\lambda_i^{\frac{p}{2}}$ is also locally integrable, so by Lemma~\ref{lemma:lpv-banach}, $L^p(\lambda_i^{\frac{p}{2}},\Omega)$ is separable.  The finite direct sum of separable spaces is separable, so $QL^p(\Omega)$ is separable.    
\end{proof}

\medskip

We can embed the weighted Lebesgue spaces $L^p(v,\Omega)$ in a finer scale of Banach function spaces, the Orlicz spaces.  Here we recall some basic properties; for complete information, see~\cite{MR1113700,MR928802} or the briefer summary in~\cite{MR2797562}.  A Young function is a function $\Phi : [0,\infty) \rightarrow [0,\infty)$ that is continuous, convex, strictly increasing, $\Phi(0)=0$, and $\Phi(t)/t\to \infty$ as $t\to \infty$.  The functions $\Phi(t)=t^p$, $1<p<\infty$, are basic examples; other examples include $\Phi(t)=t^p\log(e+t)^q$, $1<p<\infty$, $0<q<\infty$, and $\Phi(t) = \exp(t^r)-1$, $0<r<\infty$.  While  $\Phi(t)=t$ is not a Young function we may use it in the definitions below.  Generally, we will denote Young functions by the Greek letters $\Phi$, $\Psi$, and $\Theta$.  The one exception is in the definition of Sobolev inequalities (see Section~\ref{subsec:sobolev} below) where we will use the Roman letters $A$, $B$, etc.  

Given a Young function $\Phi$ and a weight $v$, we define the weighted Orlicz space $L^\Phi(v,\Omega)$ to be the measurable functions defined on $\Omega$ such that
\[ \|f\|_{L^\Phi(v,\Omega)}
= \inf\bigg\{ \lambda > 0 : \int_\Omega \Phi\bigg(\frac{|f(x)|}{\lambda}\bigg)\,dx \leq 1 \bigg\}
< \infty. \]
With this norm, $L^\Phi(v,\Omega)$ is a Banach space (see~\cite[Section~3.3, Theorem~10]{MR1113700}).  
When $\Phi(t)=t^p$, $1\leq p<\infty$, then $L^\Phi(v,\Omega)=L^p(v,\Omega)$, with equality of norms.  

Given Young functions $\Phi$ and $\Psi$, if for all $t \geq 0$, $\Phi(t)\lesssim \Psi(t)$, then $\|f\|_{L^\Phi(v,\Omega)} \lesssim \|f\|_{L^\Psi(v,\Omega)}$.  Thus, if  $\Phi(t) \approx \Psi(t)$ for all $t\geq 0$, then  $\|f\|_{L^\Phi(v,\Omega)}\approx \|f\|_{L^\Psi(v,\Omega)}$.  The underlying constants depend  only on $\Phi$ and $\Psi$.

\begin{remark} \label{remark:orlicz-convention}
   Since $\Omega$ is a bounded domain, if $v\in L^1(\Omega)$, then the norm is only dependent on the behaviour of the Young function $\Phi$ for $t$ large;  that is, the  $L^\Phi(v,\Omega)$ and $L^\Psi(v,\Omega)$ norms are equivalent provided that for some $t_0>0$, $\Phi(t) \approx \Psi(t)$ for $t\geq t_0$. In this case the implicit constant also depends on $v(\Omega)$.  Consequently, in any setting where we assume $v\in L^1(\Omega)$, by an abuse of notation, if we write $\Phi(t) \approx \Psi(t)$ we mean that this holds for $t\geq t_0>0$.
\end{remark}

The Orlicz norm has the following rescaling property:  if $\Phi$ is a Young function, and for some $r>0$,   $\Psi(t)=\Phi(t^r)$ is also a Young function, then it follows at once from the definition of the norm that 
\[ \||f|^r\|_{L^\Phi(v,\Omega)} = \|f\|_{L^{\Psi}(v,\Omega)}^r.  
\]

Given a Young function $\Phi$, we define the function $\bar{\Phi}$ by
\[ \bar{\Phi}(t) = \sup_{s>0}\{ st-\Phi(s) \}.\]
Then $\bar{\Phi}$ is also a Young function, referred to as the complementary Young function of $\Phi$.  It satisfies the pointwise relationship
\[ t \leq \Phi^{-1}(t) \bar{\Phi}^{-1}(t) \leq 2t, \qquad t \geq 0. \]
H\"older's inequality extends to the scale of Orlicz spaces:
\begin{equation} \label{eqn:orlicz-holder}
\int_\Omega |fg|\,vdx \leq 2\|f\|_{L^\Phi(v,\Omega)}\|g\|_{L^{\bar{\Phi}}(v,\Omega)}.  
\end{equation}
Note that if $\Phi(t)=t^p$, $1<p<\infty$, then $\bar{\Phi}(t)\approx t^{p'}$.

We can also define a generalized H\"older's inequality.   Given three Young functions $\Phi,\,\Psi,\,\Theta$, suppose that for all $t\geq 0$, 
\begin{equation} \label{ineq:2.4}
    \Phi^{-1}(t)\Psi^{-1}(t) \lesssim \Theta^{-1}(t).
\end{equation}
Then there exists a constant $K=K(\Phi,\Psi,\Theta)$ such that
\begin{equation}\label{eqn:gen-orlicz-holder}
\|fg\|_{L^\Theta(v,\Omega)} \leq K \|f\|_{L^\Phi(v,\Omega)} \|g\|_{L^{\Psi}(v,\Omega)}.  
\end{equation}
If inequality~\eqref{ineq:2.4} only holds for $t\geq t_0>0$ and $v\in L^1(\Omega)$, then~\eqref{eqn:gen-orlicz-holder} holds with a constant depending also on $v(\Omega)$.  See~\cite[Theorem~2.3]{MR194881}.

Below, we will use the inverses and complementary functions of two specific Young functions.  If $\Phi(t)= t^p\log(e+t)^q$, $1<p<\infty$ and $0<q<\infty$, then 
\[ \Phi^{-1}(t) \approx \frac{t^{\frac{1}{p}}}{\log(e+t)^\frac{q}{p}}, 
\qquad 
 \bar{\Phi}(t) \approx \frac{t^{p'}}{ \log(e+t)^{q(p'-1)}}. \]
 If $\Psi(t)=\exp(t^r)-1$, $r>0$, then the inverse and conjugate function are more complicated, but if we assume that $v\in L^1(\Omega)$, then for all $t\geq 1$ we have that 
\begin{equation*}
    \Psi^{-1}(t) \approx \log(e+t)^{\frac{1}{r}},
 \qquad
 \bar{\Psi}(t) \approx t\log(e+t)^{\frac{1}{r}}. \label{orlicz-inverse}
\end{equation*}

\begin{remark}
    Below, in determining specific Young functions, if $v\in L^1(\Omega)$, we will make use  of the following equivalences without comment: for any $a>0$,
\[ \log(1+t) \approx \log(e+t) \approx \log(e+t^a). \]
As in Remark~\ref{remark:orlicz-convention}, we implicitly assume that this holds for $t\geq t_0>0$.
\end{remark}

The following result is a generalization of the classical interpolation inequality in Lebesgue spaces (see~\cite[inequality~(7.9)]{MR1814364})  to the scale of Orlicz spaces.  For the proof we need a lemma that was originally proved in \cite{CUMR2025}, but we give a much simpler proof here. 

\begin{lemma}\label{ScottInterp}
Given Young functions $\Psi_1$ and $\Psi_2$, let $\lambda_1,\, \lambda_2>0$ be  such that $\Theta_1(t) = \Psi_1(t^\frac{1}{\lambda_1})$ and $\Theta_2(t) = \Psi_2(t^\frac{1}{\lambda_2})$ are Young functions, and such that
\begin{equation}\label{power-holder-condition}
\Psi_1^{-1}(t)^{\lambda_1}\Psi_2^{-1}(t)^{\lambda_2}\lesssim \Phi^{-1}(t).
\end{equation}
Then, there is a constant $K>0$ so that for every measurable $f,g$,
$$\||f|^{\lambda_1}|g|^{\lambda_2}\|_{L^\Phi(v,\Omega)} \leq K\|f\|_{L^{\Psi_1}(v,\Omega)}^{\lambda_1}\|g\|_{L^{\Psi_2}(v,\Omega)}^{\lambda_2}.$$
\end{lemma}

\begin{proof} 
Fix measurable functions  $f,\,g$.  Since $\Theta_1$ and $\Theta_2$ are Young functions with inverses 
\[ \Theta_1^{-1}(t) = \Psi_1^{-1}(t)^{\lambda_1}, \qquad  \Theta_2^{-1}(t) = \Psi_2^{-1}(t)^{\lambda_2}, \]
inequality~\eqref{power-holder-condition} is equivalent to $\Theta_1^{-1}(t)\Theta_2^{-1}(t)\lesssim \Phi^{-1}(t)$. Hence, by the  generalized H\"older's inequality~\eqref{eqn:gen-orlicz-holder} and rescaling, we get
\begin{equation*}\||f|^{\lambda_1}|g|^{\lambda_2}\|_{L^\Phi(v,\Omega)} 
\leq K \||f|^{\lambda_1}\|_{L^{\Theta_1}(v,\Omega)}\||g|^{\lambda_2}\|_{L^{\Theta_2}(v,\Omega)}\\
 = K\|f\|_{L^{\Psi_1}(v,\Omega)}^{\lambda_1}\|g\|_{L^{\Psi_2}(v,\Omega)}^{\lambda_2}.
\end{equation*}
\end{proof}

\begin{theorem}\label{scottinterp2} 
Given $\lambda\in (0,1)$, suppose that $\Psi,\Phi$ are Young functions such that
$$t^{1-\lambda}\Psi^{-1}(t)^{\lambda} \lesssim \Phi^{-1}(t).$$ 
Then, given  any $\epsilon>0$, 
\[ \|f\|_{L^\Phi(v,\Omega)} \leq K^{\frac{1}{\lambda}}\epsilon^{-\frac{1-\lambda}{\lambda}}\|f\|_{L^1(v,\Omega)} 
+ \epsilon\|f\|_{L^\Psi(v,\Omega)}. \]
\end{theorem}

\begin{proof} Since $t^{\frac{1}{1-\lambda}}$ and $\Psi(t^{\frac{1}{\lambda}})$ are Young functions, by Lemma~\ref{ScottInterp}
we have that
\[ \|f\|_{L^\Phi(v,\Omega)} = \||f|^{(1-\lambda)} |f|^{\lambda}\|_{L^\Phi(v,\Omega)}
\leq C \|f \|_{L^{1}(v,\Omega)}^{1-\lambda} \|f\|_{L^\Psi(v,\Omega)}^{\lambda}.
\]
If we now apply  Young's inequality~\eqref{eqn:youngs-ineq} with exponents $p=(1-\lambda)^{-1}$ $p'=\lambda^{-1}$ and $\eta=\epsilon/K$, we get that
\[\|f\|_{L^\Phi(v,\Omega)} 
\leq K^{\frac{1}{\lambda}}\epsilon^{-\frac{1-\lambda}{\lambda}}\|f\|_{L^1(v,\Omega)} + \epsilon\|f\|_{L^\Psi(v,\Omega)}.\]
\end{proof}

\subsection*{Degenerate Sobolev spaces}
We begin by defining our spaces of test functions.  Let  $\lip_\loc(\Omega)$ consist of functions in $\Omega$ that are Lipschitz on compact subsets of $\Omega$. 
Let $\lip_0(\Omega)\subset \lip_\loc(\Omega)$ be Lipschitz functions  whose supports are compact subsets of $\Omega$.

Fix $1\leq p<\infty$, a matrix function $Q$, and a weight $v$. Define $Q\lip_\loc(v,\Omega)$ to be the collection of all functions $\varphi \in \lip_\loc(\Omega)$ such that 
\[ \|\varphi\|_{QH^{1,p}(v,\Omega)} 
= \|\varphi\|_{L^p(v,\Omega)} + \|\nabla \varphi \|_{QL^p(\Omega)}< \infty. 
\]
Let $Q\lip_0(v,\Omega) = \lip_0(\Omega)\cap Q\lip_\loc(v,\Omega)$.
Without additional  assumptions on $Q$ and $v$, this norm need not be finite for every $\varphi \in \lip_\loc(\Omega)$ or even $\varphi \in \lip_0(\Omega)$.  But, for example, if  $v\in L^1_{\loc}(\Omega)$ and $|Q|_{\op} \in L^{\frac{p}{2}}_\loc(\Omega)$, then $\lip_0(\Omega)\subset Q\lip_\loc(\Omega)$.  We will not automatically assume these integrability conditions, but we will always implicitly assume that $v$ and $Q$ are such that 
$Q\lip_0(v,\Omega)$  is nontrivial.

We now define the degenerate Sobolev space that will be the solution space for our partial differential equations.  Our approach is originally from~\cite{MR2574880}, and was later adopted in, for example,~\cite{MR2994671,MR3846744, MR4280269}. Define the space $QH^{1,p}(v,\Omega)$ to be the abstract closure of $Q\lip_\loc(v,\Omega)$ with respect to this norm:  that is, the collection of equivalence classes of sequences that are Cauchy with respect to the $\|\cdot\|_{QH^{1,p}(v,\Omega)}$ norm.  We can identify this space with a closed subspace $\mathcal{W}$ of the Banach space $L^p(v,\Omega) \oplus QL^p(\Omega)$ in the following way.  Given an equivalence class in $QH^{1,p}(v,\Omega)$ represented by the Cauchy sequence $\{\varphi_k\}_{k=1}^\infty$ of $Q\lip_{\textrm{loc}}(v,\Omega)$ functions, the sequence $\{(\varphi_k,\nabla \varphi_k)\}_{k=1}^\infty$ is Cauchy in $L^p(v,\Omega)\oplus QL^p(\Omega)$ and thus has limit $(u,\vecg)$.  Since this limit is unique to each equivalence class, mapping an equivalence class to its associated limit defines an isometry with range $\mathcal{W}$. Define the space $QH_0^{1,p}(v,\Omega)$ to be the subspace of $QH^{1,p}(v,\Omega)$ that is the closure of $Q\lip_0(v,\Omega)$ with respect to the $\|\cdot\|_{QH^{1,p}(v,\Omega)}$ norm.  


When $p=2$, it follows from Lemmas~\ref{lemma:lpv-banach} and~\ref{lemma:Qspace} that the spaces $QH^{1,2}(v,\Omega)$ and $QH_0^{1,2}(v,\Omega)$ are Hilbert spaces.  Given $\vecu= (u,\vecg),\, \vecw= (w,\vech) \in QH^{1,2}(v,\Omega)$,  the inner product is given by
\[ \langle \vecu, \vecw \rangle_{QH^{1,2}(v,\Omega)} 
= \int_\Omega uw\,vdx + \int_\Omega  \sqrt{Q} \vecg\cdot \sqrt{Q} \vech \,dx. \]

Though it will not play a direct role in our work, for completeness we also define the space $QW^{1,p}(v,\Omega)$.
Let $W^{1,1}_\loc(\Omega)$ consist of all functions $f$ which have weak derivatives (i.e., distributional derivatives--see \cite[Section~7.3]{MR1814364}) such that $f,\,\nabla f \in L^1_\loc(\Omega)$.  The space $QW^{1,p}(v,\Omega)$ consists of all $f\in W^{1,1}_\loc(\Omega)$ such that 
\[ \|f\|_{QW^{1,p}(v,\Omega)} = \|f\|_{QH^{1,p}(v,\Omega)} <\infty. \]
If $QW^{1,p}(v,\Omega)$ is a Banach space, then $QH^{1,p}(v,\Omega)\subset QW^{1,p}(v,\Omega)$, and so every pair $(u,\vecg)$ in $QH^{1,p}(v,\Omega)$ must satisfy $\nabla u = \vecg$.   This is the case if, for instance, we impose integrability assumptions on the largest and smallest eigenvalues of $Q$.  The following result, though stated differently, was essentially proved in~\cite[Theorem~5.2]{MR3544941}.  

\begin{prop} \label{prop:QW-banach}
    Given $1\leq p<\infty$ and a matrix function $Q$, suppose there exists a weight $\tilde{w}$ such that $\tilde{w}^{-1} \in L^{\frac{p}{p'}}_\loc(\Omega)$ and for every $\xi \in \R^n$ and a.e.~$x\in \Omega$, $\tilde{w}(x)|\xi|^p \leq |\sqrt{Q(x)}\xi |^p$.  Then $QW^{1,p}(v,\Omega)$ is a Banach space. 
\end{prop}

However, we stress that this result need not hold, and for $v$ and $Q$ sufficiently degenerate, it may be the case that $\vecg$ is not the weak derivative of $u$ and may not even be uniquely determined by $u$. In particular, there exists $v$ and $Q$ such that $(0,1)\in QH^{1,p}(v,\Omega)$ (see~\cite[Section~2.1]{MR643158}).  Nevertheless, we will, in an abuse of notation, adopt the convention that given an element $(u,\vecg) \in QH^{1,p}(v,\Omega)$, we will define $\nabla u =\vecg$, and denote the $j$-th component of $\vecg$ by $\partial_j u$.  Moreover, we will often write  $(u,\nabla u) \in QH^{1,p}(v,\Omega)$, or even simply $u \in QH^{1,p}(v,\Omega)$.   

Even though they are not weak derivatives, the functions $\nabla u$ have many of the same properties if we assume that $v\in L^1(\Omega)$ and that it dominates the largest eigenvalue of $Q$.  (See~\cite[Section~2]{MR4280269}.) In particular, we will need the following truncation result.  For a proof see~\cite[Lemma~2.14]{MR4280269}.

\begin{lemma} \label{lemma:truncate}
Given $v$ and $Q$, suppose $v\in L^1(\Omega)$ and $|Q(x)|_\op  \leq kv(x)$ a.e.  Let $u\in QH^{1,2}_0(v,\Omega)$ and fix $r>0$.
If $S(r) =\{x \in\Omega : u(x)>r\}$, then $((u- r)_+, \chi_{S(r)}\nabla u)\in QH^{1,2}_0(v,\Omega)$.
\end{lemma}

\subsection*{Degenerate Sobolev and Poincar\'e inequalities}
\label{subsec:sobolev}

Key hypotheses for our main results are the existence of degenerate Sobolev and Poincar\'e inequalities on $QH^{1,2}(v,\Omega)$.  
Given $\Omega$, a weight $v$, a matrix $Q$,  $1 \leq p<\infty$, and  $\sigma\geq 1$, we say that there is a global degenerate Sobolev inequality with gain $\sigma$ (or without gain if $\sigma=1$)  on $QH_0^{1,p}(v,\Omega)$
if for every $\varphi\in Q\lip_0(\Omega)$,
\begin{equation} \label{eqn:global-sobolev} 
\bigg( \int_\Omega |\varphi|^{\sigma p} \,\,vdx\bigg)^{\frac{1}{\sigma p}}
\leq
S(p,\sigma) \bigg( \int_\Omega |\sqrt{Q}\nabla\varphi|^p \,dx\bigg)^{\frac{1}{p}}.
\end{equation}
Note that by a standard approximation argument using Fatou's lemma, the same inequality holds for any $\varphi \in QH_0^{1,2}(v, \Omega)$. Also, it is immediate from H\"older's inequality that if $v\in L^1(\Omega)$ and a global Sobolev inequality holds with gain $\sigma>1$, then it it holds for all $\tau$, $1\leq \tau < \sigma$.

\begin{remark}
    The existence of such Sobolev inequalities has been considered by several authors:  see, for instance, \cite{MR0805809,MR1354890,MR4224718}.  Recently, we have shown in~\cite{DCU-FED-SR,DCU-FED-SR-2} the existence of such Sobolev inequalities only assuming integrability assumptions on the largest and smallest eigenvalues of $Q$.
\end{remark}

We will also assume the existence of a degenerate Sobolev inequality where the gain is given in the scale of Orlicz spaces.  Let $A$ be a Young function such that $t^p \lesssim A(t)$.  Then we say that there is a global Sobolev inequality with gain $A$ if for every $\varphi\in Q\lip_0(\Omega)$,
\begin{equation} \label{eqn:global-sobolev-orlicz} 
\|\varphi\|_{L^A(v,\Omega)} \leq 
S(p,A) \bigg( \int_\Omega |\sqrt{Q}\nabla\varphi|^p \,dx\bigg)^{\frac{1}{p}}.
\end{equation}
If $t \lesssim B(t) \lesssim A(t)$ and $v\in L^1(\Omega)$,  then \eqref{eqn:global-sobolev-orlicz} implies that  there is also a global Sobolev inequality with gain $B$.  
\medskip

Another key assumption will be the existence of a weak, local degenerate Poincar\'e inequality without gain:  given any compact set $K\subset \Omega$ and $\epsilon>0$, there exists $\delta>0$ so that if $0<r<\delta$ and $x\in K$, then for all $\varphi \in Q\lip_0(v,\Omega)$,  
\begin{equation} \label{eqn:weak-poincare}
 \bigg(\int_{B(x_0,r)}|\varphi-\langle \varphi \rangle_{B,v}|^p\,vdx\bigg)^{\frac{1}{p}}
\leq \epsilon \bigg(\int_{B(x_0,r)}|\sqrt{Q}\nabla \varphi|^p
\,dx\bigg)^{\frac{1}{p}},
\end{equation}
where $\langle \varphi \rangle_{B,v} = |v(B)|^{-1}\int_B
\varphi\,vdx$.
 Such an inequality holds, for instance, if the local Poincar\'e inequality
holds with a super-radius, a concept first introduced
in~\cite{MR4224718}:  for some $0<\epsilon<1$,
\begin{equation} \label{eqn:poincare-super}
\bigg(\int_{B(x_0,r)}|\varphi-\langle \varphi \rangle_{B,v}|^p\,vdx\bigg)^{\frac{1}{p}}
\leq Cr^{\epsilon} \bigg(\int_{B(x_0,r)}|\sqrt{Q}\nabla \varphi|^p
\,dx\bigg)^{\frac{1}{p}}.
\end{equation}

In fact, we can assume a somewhat more general condition:  that our local  Poincar\'e inequality is defined using balls from a different geometric structure, one adapted to the degeneracies of the equation.    To state our hypotheses, we need some additional definitions.  Let $\rho$ be a quasimetric on a bounded open set $\Theta$: that is, a map $\rho : \Theta \times \Theta \rightarrow [0,\infty)$ such that there exists a constant $k>0$ such that for all $x,\,y,\,z \in \Theta$,
\begin{enumerate}
    \item $\rho(x,y)\geq 0$ and $\rho(x,y)=0$ if and only if $x=y$;
    \item $\rho(x,y)=\rho(y,x)$;
    \item $\rho(x,y) \leq k(\rho(x,z)+\rho(z,y))$.
\end{enumerate}
Given $x\in \Theta$ and $r>0$ define the $\rho$-ball 
$B_\rho(x,r) = \{ y \in \Theta : \rho(x,y) < r \}$.
Given a quasimetric $\rho$ on a bounded, open set $\Theta$, we will assume that
    the $\rho$-balls are open in the standard topology on $\rn$, and we will assume that
for each $x\in\Theta$, there exists $\delta(x)>0$ such that if $0<r<\delta(x)$, then $\overline{B_\rho(x,r)} \subset \Theta$. (Such balls can be thought of as being "far" from the boundary of $\Theta$.)  Finally, we will assume that 
$\rho$ satisfies a local, geometric doubling condition:  there exists a function $G : (0,\infty)\rightarrow (0,\infty)$ such that, given any compact set $K\subset \Theta$, there exists $\delta>0$ such that if $0<s<r<\delta$ and $x\in K$, there exist at most $G(r/s)$ points $x_i\in B_\rho(x,r)$ such that the balls $B_\rho(x_i,s)$ are disjoint.

\begin{remark} 
Given any quasimetric $\rho$, there exists another, equivalent quasimetric $\rho_\#$ such that all $\rho_\#$-balls are open.  This metric also has the property that there exists $\beta>0$ such that $\rho_\#^\beta$ is a metric; this implies that  $\rho_\#$-balls also have the second property.  In practice we can usually replace $\rho$ by $\rho_\#$, and so guarantee these properties.  For a careful discussion of these properties of quasimetrics, see~\cite[Section~2.1]{MR3310009}.  For a discussion of the property that balls can be taken to be far from $\partial \Theta$, see~\cite[Remark~3.6]{MR3095112}. 
\end{remark}


We can now state our local degenerate Poincar\'e inequality:  we will assume that there exists $c\geq 1$ such that, given any compact set $K\subset \Omega$ and $\epsilon>0$, there exists $\delta>0$ so that if $0<r<\delta$ and $x\in K$, then for all $\varphi \in Q\lip_0(v,\Omega)$,  
\begin{multline}\label{eqn:local-poincare-QM}
\bigg(\int_{B_\rho(x,r)} |\varphi(y)-\varphi_{v,B_\rho}|^p \,vdy \bigg)^{\frac{1}{p}} \\
\leq \epsilon \bigg[
\bigg(\int_{B_\rho(x,cr)} |\sqrt{Q}\nabla \varphi|^p \,dy\bigg)^{\frac{1}{p}}
+ \bigg(\int_{B_\rho(x,cr)} |\varphi|^p \,vdy\bigg)^{\frac{1}{p}} \bigg]
= \epsilon\|\varphi\|_{QH^{1,p}(v,B_\rho(x,cr))},
\end{multline}
where $\varphi_{v,B_\rho}$ is the weighted average~\eqref{eqn:wtd-avg}.
For a discussion of this Poincar\'e inequality, see~\cite[Section~2]{MR4069608}.

\begin{remark}
All of these hypotheses hold if we take $\rho$ to be the standard Euclidean distance, so that it is sufficient to assume that~\eqref{eqn:weak-poincare} or~\eqref{eqn:poincare-super} hold for Euclidean balls.  
    Recently, we showed in~\cite[Section~5]{DCU-FED-SR-2} that~\eqref{eqn:poincare-super} holds given suitable integrability assumptions on the largest and smallest eigenvalues of $Q$.  
\end{remark}

\begin{remark}
    For specific examples of settings where these Sobolev and Poincar\'e inequalities hold, see~\cite{MR4224718,MR643158,MR0805809,DCU-FED-SR}.  We also note that in~\cite{KRSS2024}, the existence of a local degenerate Poincar\'e inequality played an important role in the proof. 
\end{remark}

\begin{remark}
    The assumption that the quasimetric satisfies a local geometric doubling condition is needed in order to apply the Poincar\'e inequality to prove that the projection mapping from $QH_0^{1,p}(v,\Omega)$ to $L^p(v,\Omega)$ is compact.  Our proof requires the construction of a covering of balls with finite overlap, which uses this condition.  (See Lemmas~\ref{lemma:compact-gain} and~\ref{lemma:compact-no-gain}.)  It is an open question whether we can prove compactness without assuming geometric doubling.  
\end{remark}

\subsection*{Degenerate elliptic equations with lower order terms}
Here we define the  degenerate, second order, linear equation that we will consider, and then give our definition of a solution.  
We are interested in second order elliptic equations with lower order terms, which we will write as
\begin{equation} \label{eqn:formal-eqn}
  Lu = -v^{-1}\Div(Q\nabla u)  + \bH\cdot \bR u + \bS'\bG u + Fu = f + \bT' \vecg.  
\end{equation}
Initially, we only assume that $v$ is a measurable weight function, and $Q$ is a real, self-adjoint, $n\times n$ matrix function that is positive semidefinite and whose entries are finite a.e.  However, in dealing with this equation with non-zero first order terms, we will be required to assume that $v\in L^1(\Omega)$ and $|Q(x)|_\op \leq kv(x)$ a.e.  We will carefully track this assumption.

For the remaining coefficients, we will assume that  $f,\, F \in L^2(v,\Omega)$.  For some fixed $N\in \N$, we assume that $\bG,\, \bH,\, \vecg \in L^2(v,\Omega, \R^N)$. We will denote the components of these functions by $g_i,\,G_i,\, H_i$, $1\leq i \leq N$. 

The operators  $\bR,\, \bS,\, \bT$ are $N$-tuples of  degenerate subunit vector fields.
To define them, recall that a vector field is
an operator of the form  $V(x)=\vecv(x) \cdot \nabla$, where $\vecv$ is an $\R^n$-valued function.  We say that $V$ is a degenerate subunit  vector field (with respect to $v$ and $Q$) if
    \begin{equation} \label{eqn:norm-subunit}
        \|Vu\|_{L^2(v,\Omega)} \leq C(V)\|u\|_{QH^{1,2}(v,\Omega)}.
    \end{equation}  
This norm inequality holds if we assume the following stronger pointwise inequality:   for all $\xi \in \R^n$,
\begin{equation} \label{subunicity}
\langle \vecv, \xi \rangle^2
\leq 
v^{-1}|\sqrt{Q} \xi|^2 = v^{-1}\langle Q\xi,\xi \rangle.    
\end{equation}
Inequality~\eqref{subunicity} is adapted from~\cite{MR2204824}, where it is assumed that $v=1$. (See also~\cite{MR1488238,MR0730094}.)  With the most general hypotheses, we need this stronger pointwise condition to ensure that our definition of a weak solution is well-defined;   see Remark~\ref{rem:subunit-defn}.  However, with the stronger assumptions we need to prove existence and uniqueness, it suffices to assume inequality~\eqref{eqn:norm-subunit}.  We apply it using 
Lemma~\ref{lemma:subunicity-cond} below;  also see Remark~\ref{remark:subunit-norm-cond}.

\medskip

Let
$\bR = (R_1,\ldots,R_N)$, $\bS=(S_1,\ldots S_N)$, and $\bT=(T_1,\ldots,T_N)$  be $N$-tuples of degenerate subunit vector fields.  Define the components of each $R_i$,  $1\leq i \leq N$, by
\[ R_i(x) = \vecr_i(x) \cdot \nabla = \sum_{j=1}^n r_{ij}(x)\partial_j.\]
The components of each $S_i$, denoted $s_{ij}$, and $T_i$, denoted $t_{ij}$ are defined similarly.  
For each $i$, define $S_i'$ to be the formal adjoint of $S_i$:
\[ S_i'(x)u = - v^{-1}\Div(\vecs_i uv) = - v^{-1}\Div(s_{i1}(x)uv,\ldots,s_{in}(x)uv). \]
Define the adjoints $T_i'$ in the same way.  
Define
\begin{gather*}
\bH \cdot \bR u = \sum_{i=1}^N H_iR_iu, \qquad 
 \bS' \bG u = \sum_{i=1}^N S_i'(G_iu),\\
 \bG \cdot \bS\varphi  = \sum_{i=1}^N G_i S_i \varphi,  \qquad
 \vecg \cdot \bT \varphi = \sum_{i=1}^N g_i   T_i \varphi. 
\end{gather*}

\begin{remark}
    We could, superficially, have a more general equation by letting $\bR,\,\bS,\,\bT$ be tuples of different lengths.  However, we could extend any tuple to a longer one by adding zeros. Therefore, for simplicity of notation we assume that they are all the same length.
\end{remark}

Hereafter, we will assume the following compatibility condition on the coefficients:
\begin{equation} \label{eqn:compatible}
        \int_\Omega (\bG\cdot\bS \varphi +F \varphi)\,\,vdx\geq 0
    \end{equation}
for every nonnegative $\varphi \in Q\lip_0(v,\Omega)$.  When there are no first order terms (that is, $\bG$ and $\bS$ are zero), this reduces to the assumption that $F$ is non-negative.  Even in the classical case when $Q$ is uniformly elliptic and $v=1$, this assumption is necessary to prove uniqueness of solutions:  see~\cite[Corollary~3.2]{MR1814364}.  In our proofs, inequality~\eqref{eqn:compatible} is required to show the uniqueness of the zero solution:  see Section~\ref{section:zero}.  This in turn is essential to apply the Fredholm alternative to prove the existence and uniquenss of solutions:  see Section~\ref{section:main-proofs}.

\begin{remark}
    Though we have not considered it, the question of  the existence of (multiple) solutions when we do not assume~\eqref{eqn:compatible} is an interesting one.
\end{remark}

\begin{remark}
If we did not define our operator $L$ with a negative sign in front of the second order term, \eqref{eqn:compatible} would have the inequality reversed.  In this case the condition is sometimes referred to as a negativity condition and corresponds to the assumption on the coefficients  often made in the classical case:  see, for instance~\cite[Section~8.1]{MR1814364} or~\cite[Remark~2.1]{MR2994671}.
\end{remark}

When \eqref{eqn:compatible} is used, it will be applied to an element $(h,\nabla h)\in QH^{1,2}_0(v,\Omega)$ where $h\geq 0$ $v$-almost everywhere in $\Omega$.  This is possible since we are able to approximate such elements with non-negative $Q\lip_0(v,\Omega)$ functions. The proof requires the assumption that $v$ is integrable and dominates $|Q|_\op$.

\begin{lemma} \label{lemma:non-neg-test}
Suppose $v\in L^1(\Omega)$ and $|Q(x)|_\op \leq kv(x)$ a.e.  Given $(h,\nabla h)\in QH^{1,2}_0(v,\Omega)$ with $h\geq 0$ $v$-almost everywhere, there exists a sequence $\{\psi_k\}_{k=1}^\infty$ in $Q\lip_0(v,\Omega)$ such that for each $k$,  $\psi_k\geq 0$ in $\Omega$ and
the sequence $\{\psi_k\}_{k=1}^\infty$ converges to $(h,\nabla h)$ in $QH^{1,2}_0(v,\Omega)$.
\end{lemma}

\begin{proof} Fix $(h,\nabla h)\in QH^{1,2}_0(v,\Omega)$ with $h(x)\geq 0$ $v$-a.e. For each $k\in\N$, Let $\rho_k \in C^1(\R)$ satisfy 
\[ \rho_k(x) =
\begin{cases}
    -\dfrac{1}{k}, & x\leq -2\\
    x, & 0\leq x\leq k, \\
    k+1, & x \geq k+2, \\
\end{cases}
\]
and $0\leq \rho_k'(x)\leq 1$.  Thus, $\rho_k,~\rho_k' \in L^\infty(\R)$.   Now define
 \[ H_k = \rho_k(h), \qquad  \nabla H_k = \rho_k'(h)\nabla h. \]
We claim that $(H_k,\nabla H_k) \in QH^{1,2}_0(v,\Omega)$.  The proof of this is implicit in~\cite[Lemma~2.18]{MR4280269}; for completeness we give the details.  Since $(h,\nabla h) \in QH^{1,p}_0(v,\Omega)$, there exists a sequence $\{\varphi_j \}_{j=1}^\infty$ in $Q\lip_0(v,\Omega)$ such that $\varphi_j \to h$ in $L^2(v,\Omega)$ and $\nabla \varphi_j \to \nabla h$ in $QL^2(\Omega)$.  If we pass to a subsequence, we may assume that $\varphi_j \to h$ pointwise $v$-a.e. and $\sqrt{Q}\nabla \varphi_j \to \sqrt{Q}\nabla h$~a.e.   

Since $-\frac{1}{k} \leq \rho_k(\varphi_j), \, \rho_k(h) \leq k+1$, and since $v\in L^1(\Omega)$, by the dominated convergence theorem, $\rho_k(\varphi_j)\to \rho_k(h)$ in $L^2(v,\Omega)$. Similarly, $\sqrt{Q}\nabla \rho_k(\varphi_j) = \rho_k'(\varphi_j) \sqrt{Q}\nabla \varphi_j \to \rho_k'(h) \sqrt{Q} \nabla h$, and 
\begin{multline*} 
|\sqrt{Q}\nabla \rho_k(\varphi_j) - \rho_k'(h) \sqrt{Q} \nabla h|^2 \\
\leq 2\big(  |\sqrt{Q}\nabla \rho_k(\varphi_j)|+ |\rho_k'(h) \sqrt{Q} \nabla h|\big)^2 
\to 2|\rho_k'(h) \sqrt{Q} \nabla h|^2.
\end{multline*}
Since the right-hand term is in $L^1(\Omega)$, by the generalized dominated convergence theorem, we have that $\nabla \rho_k(\varphi_j)\to \rho_k'(h)  \nabla h$ in $QL^2(\Omega)$.  Thus, $(H_k,\nabla H_k) \in QH^{1,p}_0(v,\Omega)$ and $H_k \in L^\infty(\Omega)$.

%

\smallskip

We now claim that $(H_k,\nabla H_k)$ converges to $(h,\nabla h)$ in $QH_0^{1,2}(v,\Omega)$. 
First, since $0\leq h(x)<\infty$ $v$-a.e., $H_k(x)=h(x)$ for $k$ sufficiently large; thus  $H_k \to h$ $v$-a.e.  Since we also have $H_k(x) \leq h(x)+1$ and $v\in L^1(\Omega)$, by the  dominated convergence theorem we have that $H_k\to h$ in $L^2(v,\Omega)$.
To see the convergence of the gradients, note that we  again have that for almost every $x$, $\sqrt{Q(x)}\nabla H_k(x) =\sqrt{Q(x)}\nabla h(x)$ for $k$ sufficiently large .  Since $0\leq \rho_k' \leq 1$, $|\sqrt{Q(x)}\nabla H_k(x)| \leq |\sqrt{Q(x)}\nabla h|$.  Hence, by the dominated convergence theorem, $\nabla H_k\to \nabla h$ in $QL^2(\Omega)$.
Thus, we have shown that 
$\{(H_k,\nabla H_k)\}_{k=1}^\infty$ converges to $(h,\nabla h)$ in $QH^{1,2}_0(v,\Omega)$.  

To complete the proof, since $H_k$ is bounded,  by \cite[Lemma~2.15]{MR4280269} (which holds because of our assumptions on $v$ and $Q$) we have that for each $k$ there exists a function $\psi_k\in Q\lip_0(v,\Omega)$ such that
\begin{enumerate}
\item $0\leq \psi_k(x)\leq k+2$ $v$-a.e., 
\item $\|\psi_k-H_k\|_{L^2(v,\Omega)}<\tfrac{1}{k}$, and 
\item $\|\nabla \psi_k-\nabla H_k\|_{QL^2(\Omega)}<\tfrac{1}{k}$.
\end{enumerate}
 Therefore, for any $k\in \N$,
\begin{equation*}
\|(\psi_k,\nabla \psi_k) - (h,\nabla h)\|_{QH^{1,2}(v,\Omega)} 
\leq \|(H_k,\nabla H_k) - (h,\nabla h)\|_{QH^{1,2}(v,\Omega)} +\dfrac{2}{k}.  
\end{equation*}
Since the right-hand side tends to $0$ as $k\to \infty$, we see that $\{\psi_k\}_{k=1}^\infty$ is the desired sequence.
\end{proof}

\smallskip

We can now define what is meant by a solution of~\eqref{eqn:formal-eqn}.

\begin{defn} \label{defn:dirichlet}
Given a weight $v$ and a matrix $Q$ and the operator $L$ defined above, we  say that the pair $(u,\nabla u) \in QH_0^{1,2}(v, \Omega)$ is a degenerate weak solution to the Dirichlet problem
\begin{equation} \label{eqn:dirichlet} 
\begin{cases}
Lu = f + \bT' \vecg, & x \in \Omega, \\
 u = 0, & x \in \partial \Omega,
 \end{cases}
\end{equation}
if for all $\varphi \in Q\lip_0(v,\Omega)$,
\begin{multline} \label{eqn:weak-soln}
\int_\Omega   Q\nabla u \cdot \nabla \varphi \, dx
+ \int_\Omega \bH\cdot \bR u \, \varphi \, v dx 
+ \int_\Omega u\, \bG \cdot \bS \varphi \,vdx
+ \int_\Omega F u \varphi \,\,vdx  \\
= \int_\Omega  f\varphi \,vdx  + \int_\Omega \vecg \cdot \bT\varphi \,vdx. 
\end{multline}
\end{defn}

\begin{remark} \label{rem:subunit-defn}
    Implicit in this definition is the assumption that every integral in~\eqref{eqn:weak-soln} is finite.  With the $L^2$ integrability assumptions on the coefficient functions and the norm subuniticity condition~\eqref{eqn:norm-subunit} on the vector fields, this follows by H\"older's inequality for every integral except the integral of $u\bG \cdot \bS \varphi v$.  To show that this integral is finite, however, we need to assume  that the stronger pointwise subuniticity condition~\eqref{subunicity} holds, and  that $|Q(x)|_\op \leq kv(x)$ a.e.  In this case, we can apply inequality~\eqref{eqn:subunit-pointwise} below  (which is an immediate consequence of~\eqref{subunicity}) to get
    \begin{multline*}
    \bigg| \int_\Omega u\, \bG \cdot \bS \varphi \,vdx \bigg|
    \leq \int_\Omega |u||\bG||\bS \varphi|\, vdx \\
    \leq \sqrt{N} \int_\Omega |u||\bG| v^{-\frac{1}{2}}|Q|_\op^{\frac{1}{2}} |\nabla \varphi| \,vdx
    \leq \sqrt{N}k \|u\|_{L^2(v,\Omega)}\|\bG\|_{L^2(v,\Omega)}\|\nabla \varphi\|_{L^\infty(\Omega)}
    < \infty. 
    \end{multline*}
    
    But, as we will show below (see Lemma~\ref{lemma:QH1-test-func}), with the stronger assumptions that we will make on the coefficients, each of these integrals is finite only assuming the norm condition~\eqref{eqn:norm-subunit}.  
\end{remark}

\medskip

\section{Statement of main results}
\label{section:main-results}

To state our main results, we first gather together all of the assumptions that we outlined above and that are implicit in Definition~\ref{defn:dirichlet}.  We will always assume $v$ is a measurable weight function, and $Q$ is a real, self-adjoint, $n\times n$ matrix function that is positive semidefinite and whose entries are finite a.e. However, we will specify in each result the additional assumptions on $v$ and $Q$ that we will need.

\begin{hyp} \label{hyp:global-hyp}
    We make the following assumptions:  
    \begin{itemize}
        \item $\Omega \subset \R^n$ is a bounded, connected, open domain.

        \item $f,\, F \in L^2(v,\Omega)$, and $\bG,\, \bH,\, \vecg \in L^2(v,\Omega,\R^N)$ for some $N\in \N$.

        \item $\bR,\, \bS,\, \bT$ are $N$-tuples of degenerate vector fields satisfying~\eqref{eqn:norm-subunit}.

        \item $F,\, \bG, \bS$ satisfy the compatibility condition~\eqref{eqn:compatible}.  

    \end{itemize}
\end{hyp}

We also recall all of our assumptions needed for a Poincar\'e inequality.

\begin{hyp} \label{hyp:poincare}
    Given a quasimetric $\rho$ on a bounded, open set $\Theta$, we will assume that
    \begin{itemize}
    \item The $\rho$-balls are open in the standard topology on $\rn$. 

    \item For each $x\in\Theta$, there exists $\delta(x)>0$ such that if $0<r<\delta(x)$, then $\overline{B_\rho(x,r)} \subset \Theta$. 

    \item The quasimetric $\rho$ satisfies a local, geometric doubling condition:  there exists a function $G : (0,\infty)\rightarrow (0,\infty)$ such that, given any compact set $K\subset \Theta$, there exists $\delta>0$ such that if $0<s<r<\delta$ and $x\in K$, there exist at most $G(r/s)$ points $x_i\in B_\rho(x,r)$ such that the balls $B_\rho(x_i,s)$ are disjoint.
\end{itemize}
\end{hyp}

Our first result assumes the existence of a global Sobolev inequality with gain in the scale of Lebesgue spaces. In this case, we only need to assume the existence of a Poincar\'e inequality on $\Theta=\Omega$. 

\begin{theorem}\label{thm:main-theorem}
Suppose that $v$ are $Q$ are such that $v\in L^1(\Omega)$ and $|Q(x)|_\op \leq kv(x)$ a.e. Given  Hypothesis~\ref{hyp:global-hyp}, suppose that the global degenerate Sobolev inequality~\eqref{eqn:global-sobolev} holds with gain $\sigma>1$.  Also suppose the pair $(\Omega, \rho)$  is a quasimetric space that satisfies Hypothesis~\ref{hyp:poincare} and the local  degenerate Poincar\'e inequality~\eqref{eqn:local-poincare-QM}  holds on $\Omega$. If for some $q>2\sigma'$, $\bH,\,\bG \in L^{q}(v,\Omega,\R^N)$  and $F\in L^{\frac{q}{2}}(v,\Omega)$,  then the Dirichlet problem~\eqref{eqn:dirichlet} has a unique degenerate weak solution $ u \in QH_0^{1,2}(v,\Omega) $.
\end{theorem}

Our second result assumes the existence of a global degenerate Sobolev inequality on $\Omega$ in the scale of Orlicz spaces.  

\begin{theorem}\label{thm:log-gain}
Suppose that $v$ are $Q$ are such that $v\in L^1(\Omega)$ and $|Q(x)|_\op \leq kv(x)$ a.e. Given  Hypothesis~\ref{hyp:global-hyp}, suppose that the global degenerate Sobolev inequality~\eqref{eqn:global-sobolev-orlicz} holds with $A(t)=t^2\log(e+t)^\sigma$, $\sigma>0$.  Also suppose the pair $(\Omega,\rho)$ is a quasimetric space  that satisfies Hypothesis~\ref{hyp:poincare} and the local  Poincar\'e inequality~\eqref{eqn:local-poincare-QM}  holds on $\Omega$. 
If for some $0<\lambda<1$, 
$\bH,\,\bG \in L^{C}(v,\Omega,\R^N)$, and $F\in L^{D}(v,\Omega)$, where $C(t)=\exp{(t^\frac{2}{\lambda\sigma})}-1$ and $D(t)=\exp{(t^\frac{1}{\lambda\sigma})}-1$, then the Dirichlet problem~\eqref{eqn:dirichlet} has a unique degenerate weak solution $ u \in QH_0^{1,2}(v,\Omega) $.
\end{theorem} 

Our next two results only assume the existence of a global degenerate Sobolev inequality without gain.  We can prove a version of the previous two theorems but only with stronger hypotheses.  First we need a size restriction on $\bG$ and $\bH$; in other words, we have a perturbation result.  Second, we also need a local degenerate Poincar\'e inequality on a larger domain $\Theta$. 

\begin{theorem}\label{thm:no-gain}
Suppose that $v$ are $Q$ are such that $v\in L^1(\Omega)$ and $|Q(x)|_\op \leq kv(x)$ a.e. Given  Hypothesis~\ref{hyp:global-hyp}, suppose that the global degenerate Sobolev inequality~\eqref{eqn:global-sobolev} holds without gain (i.e., $\sigma=1$). Suppose there exists a bounded, open, connected set $\Theta$ such that $\overline{\Omega} \subset \Theta$ and the pair $(\Theta,\rho)$ is a quasimetric space  that satisfies Hypothesis~\ref{hyp:poincare} and the local  Poincar\'e inequality~\eqref{eqn:local-poincare-QM}  holds on $\Theta$. 
If $F\in L^{\infty}(v,\Omega)$ and $\bH,\,\bG \in L^{\infty}(v,\Omega,\R^N)$, and if 
\[ \max(C(\bR),C(\bS))S(2,1))\big( \|\bG\|_{L^\infty(v,\Omega)} + \|\bH\|_{L^\infty(v,\Omega)}\big)<1, \]
then the Dirichlet problem~\eqref{eqn:dirichlet} has a unique solution $ u \in QH_0^{1,2}(v,\Omega) $.  
\end{theorem}

If we assume that our equation has no first order terms, then we do not have to assume the existence of a local degenerate Poincar\'e inequality, and we do not have to assume that $v$ is integrable and dominates $|Q|_\op$.  

\begin{theorem} \label{thm:no-gain-no-first-order}
Given  a bounded domain $\Omega$ a weight $v$, and a matrix $Q$ suppose that the global degenerate Sobolev inequality~\eqref{eqn:global-sobolev} holds without gain (i.e., $\sigma=1$). If \hspace{0.02cm} $\bH=\bG=0$,  $F\in L^{\infty}(v,\Omega)$, and $F$ is nonnegative, then the Dirichlet problem~\eqref{eqn:dirichlet} has a unique degenerate weak solution $ u \in QH_0^{1,2}(v,\Omega) $.
\end{theorem} 

\medskip

As a corollary to our main theorems above, we can prove the existence of solutions to the Dirichlet problem with nonzero boundary data.  To state this result we define boundary equality in a weak sense, and then define the equation we are interested in and its degenerate weak solutions.

\begin{defn} 
Given $f,h\in QH^{1,2}(v,\Omega)$, we say that $f=h$ on the boundary $\partial\Omega$ if and only if $(f-h)\in QH^{1,2}_0(v,\Omega)$.  
\end{defn}

\begin{remark}
Note that with this definition, we do not assume, {\em a priori}, that $f$ and $h$ are actually defined outside of $\Omega$; the problem of boundary values and traces for degenerate elliptic operators is quite difficult and we will not consider it here.    
\end{remark}

\begin{defn}
Given  $\varphi\in QH^{1,2}(v,\Omega)$, we say that $u\in QH^{1,2}(v,\Omega)$ is a degenerate weak solution of the Dirichlet problem with boundary data $\varphi$, 
\begin{equation} \label{inhomogeneous} 
\begin{cases}
Lu = f + \bT' \vecg, & x \in \Omega, \\
 u = \varphi, & x \in \partial \Omega,
 \end{cases}
\end{equation}
if $w = u-\varphi$ is a degenerate weak solution of the  Dirichlet problem 
\begin{equation} \label{inhomogeneou2} 
\begin{cases}
Lw = f + \bT' \vecg - L\varphi, & x \in \Omega, \\
 w = 0, & x \in \partial \Omega,
 \end{cases}
\end{equation}
\end{defn}

As an immediate consequence of the results above, we have the following corollary.

\begin{cor} \label{cor:boundary-data}
    Suppose that the hypotheses of Theorem~\ref{thm:main-theorem}, Theorem~\ref{thm:log-gain}, Theorem~\ref{thm:no-gain}, or Theorem~\ref{thm:no-gain-no-first-order} hold.  Then in every case the Dirichlet problem~\eqref{inhomogeneous} has a unique solution  $u \in QH^{1,2}(v,\Omega) $.  
\end{cor}

\section{Preliminaries on vector fields}
\label{section:prelim-vector}

In this section we establish some preliminary results on the operator $L$ and solutions to the Dirichlet problem.  We begin with a lemma about degenerate subunit vector fields.  

\begin{lemma} \label{lemma:subunicity-cond}
    Let $\bR=(R_1,\ldots,R_N)$ be an $N$-tuple of degenerate subunit vector fields that satisfy~\eqref{eqn:norm-subunit}.   Then for any $u\in QH_0^{1,2}(v,\Omega)$ and a.e. $x\in \Omega$ we have the norm estimate
    \begin{equation} \label{eqn:subunit-norm-cond}
\|\bR u\|_{L^2(v,\Omega)} \leq   C(\bR)\|u\|_{QH^{1,2}(v,\Omega)},
\end{equation}
    where $C(\bR)= [C(R_1)+\cdots+C(R_N)]^{\frac{1}{2}}$.  
    If the vector fields satisfy the stronger pointwise inequality~\eqref{subunicity}, then 
    \begin{equation} \label{eqn:subunit-pointwise}
    |\bR u(x)|\leq \frac{\sqrt{N}}{\sqrt{v}}{|\sqrt{Q}\nabla u(x)|},
    \end{equation}
    which in turn implies that~\eqref{eqn:subunit-norm-cond} holds with $C(\bR)=\sqrt{N}$.
\end{lemma}

\begin{remark} \label{remark:subunit-norm-cond}
    If $v$ and $Q$ satisfy the global  Sobolev inequality~\eqref{eqn:global-sobolev} with no gain (i.e., $\sigma= 1$), \eqref{eqn:subunit-norm-cond} implies that
    \[ \|\bR u\|_{L^2(v,\Omega)} \leq  C(\bR)(1+S(2,1)) \|\nabla u\|_{QL^2(\Omega)}. \]
\end{remark}

\begin{proof}
 Fix $u \in QH_0^{1,2}(v,\Omega)$.
    Then by inequality~\eqref{eqn:norm-subunit}, 
    \[ \int_\Omega |\bR u|^2\,vdx = \sum_{i=1}^N \int_\Omega |R_i u|^2\,vdx \leq \sum_{i=1}^N C(R_i)^2 \|u\|_{QH^{1,2}(v,\Omega)}^2.
\]
   Similarly, if~\eqref{subunicity} holds, then for almost every $x\in \Omega$,
\[     |\bR u(x)|=\bigg(\sum_{j=1}^N (\vecr(x) \cdot \nabla u(x))^2 \bigg)^{\frac12}
    \leq\bigg(\sum_{j=1}^N v^{-1}{|\sqrt{Q}\nabla u(x)|}^2\bigg)^{\frac12}
    \leq\frac{\sqrt{N}}{\sqrt{v}}{|\sqrt{Q}\nabla u(x)|}. \]
\end{proof}

\medskip

Our next result establishes a product rule for degenerate subunit vector fields.  We need to prove three versions of this result, corresponding to the Sobolev inequality that we assume.  

\begin{lemma}\label{product-rule-all-sobolevs}
    Suppose $v\in L^1(\Omega)$.  Let $ (u, \nabla u), (w, \nabla w) \in  QH^{1,2}_0(v, \Omega)$ and let $S=\textbf{s}\cdot\nabla$ be a degenerate subunit vector field that satisfies~\eqref{eqn:norm-subunit}. Then each of the following hold:
    \begin{enumerate}
    \item If the global Sobolev inequality \eqref{eqn:global-sobolev} holds with gain $\sigma>1$, then  there exists a unique $L^{\frac{2 \sigma}{\sigma + 1}}(v, \Omega)$ function $S(uw)$ such that $S(uw) = uS(w) + wS(u)$ (with equality in $L^{\frac{2 \sigma}{\sigma + 1}}(v, \Omega)$).
    
    \item If the global Orlicz-Sobolev inequality \eqref{eqn:global-sobolev-orlicz} holds with gain $A(t)=t^2\log(e+t)^\sigma$, then there exists a unique $L^B(v, \Omega)$ function $S(uw)$ such that $S(uw) = uS(w) + wS(u)$ in $L^B(v, \Omega)$, where  $B(t) = t\log(e+t)^\frac{\sigma}{2}$.  
    
    \item If the global Sobolev inequality \eqref{eqn:global-sobolev} holds with no gain (i.e., $\sigma=1$), then there exists a unique $L^1(v, \Omega)$ function $S(uw)$ such that $S(uw) = uS(w) + wS(u)$ in $L^1(v, \Omega)$.
   
    \end{enumerate}
\end{lemma}

\begin{remark}
    A version of part (1) of Lemma~\ref{product-rule-all-sobolevs} when $v=1$ was proved in~\cite[Lemma~3.3]{MR2994671}; our proof is adapted from this one.  Also note that though this result is only used to prove results where the equation has lower order terms, and so we are assuming $v\in L^1(\Omega)$ and $|Q(x)|_\op \leq kv(x)$ a.e., we do not need this second hypothesis to prove Lemma~\ref{product-rule-all-sobolevs}.
\end{remark}
\begin{proof}
Fix  $ (u, \nabla u), (w, \nabla w) \in  QH^{1,2}_0(v, \Omega)$ and let $\{u_j\}_{j=1}^\infty$ and $\{w_j\}_{j=1}^\infty$ be sequences in $Q\lip_0(v,\Omega)$ that converge to $u$ and $w$ in $QH^{1,2}_0(v, \Omega)$.  Then by inequality~\eqref{eqn:norm-subunit}, $Su_j\to Su$ in $L^2(v,\Omega)$, since
\[ \|S(u_j)-S(u)\|_{L^2(v,\Omega)} \leq C(S) \|u_j-u\|_{QH^{1,2}(v, \Omega)}.     \]
For each $j$,
since $u_j,\, w_j \in Q\lip_0(v,\Omega)$, by Rademacher's theorem,
$\nabla u_j,\, \nabla w_j$ equal the classical gradients of $u_j,\,w_j$ a.e., so we have that $S(u_jw_j) = u_jS(w_j) + w_jS(u_j)$ almost everywhere.

\medskip

We now consider each of the cases outlined above in turn.  Suppose first that $(1)$ holds. 
Let $\gamma = \frac{2 \sigma}{\sigma +1}$;  we will first show that the sequence $\{S(u_jw_j)\}_{j=1}^\infty$ is Cauchy in $L^\gamma(v,\Omega)$.  If we add and subtract $u_iS(w_j)+w_iS(w_j)$, we get that 
    \begin{align}\label{eqn:zero-1.0}
        \|S(u_iw_i) - S(u_jw_j)\|_{L^\gamma(v,\Omega)} & \leq \|u_iS(w_i - w_j)\|_{L^\gamma(v,\Omega)} + \|w_iS(u_i - u_j)\|_{L^\gamma(v,\Omega)} \nonumber\\
        &\qquad + \|(u_i-u_j)S(w_j)\|_{L^\gamma(v,\Omega)} + \|(w_i-w_j)S(u_j)\|_{L^\gamma(v,\Omega)} \nonumber \\
        & = I_1+I_2+I_3+I_4.
    \end{align}
We first estimate $I_1$.  By  H\"older's inequality with exponents $\sigma + 1$, $\frac{\sigma + 1}{\sigma}$, the Sobolev inequality \eqref{eqn:global-sobolev}, and~\eqref{eqn:norm-subunit},
\begin{align}\label{eqn:zero-1.1}
    I_1^\gamma 
     &  \leq \bigg(\int_\Omega |u_i|^{2\sigma}\,\,vdx\bigg)^{\frac{1}{(\sigma +1)}}\bigg( \int_\Omega \big|S(w_i-w_j) \big|^2 v \, dx \bigg)^\frac{\sigma}{\sigma+1} \\
    &  \leq S(2,\sigma)^\gamma\|u_i\|^\gamma_{QH^{1,2}(v, \Omega)} \|S(w_i-w_j)\|^\gamma_{L^2(v,\Omega)} \notag \\
    & \leq  S(2,\sigma)^\gamma C(\bS)^\gamma\|u_i\|^{\gamma}_{QH^{1,2}(v, \Omega)} \|w_i-w_j\|_{QH^{1,2}(v,\Omega)}^{\gamma}. \notag
\end{align}
Since the $w_j$ converge in $QH^{1,2}_0(v, \Omega)$, they are Cauchy.  Since the $u_i$ converge, they are uniformly bounded.  Hence $I_1$ tends to $0$ as $i,\,j \to \infty$.  The same or similar estimates hold for $I_2,~I_3$, and $I_4$, possibly exchanging the roles of $u$ and $w$.  Therefore, we have that 
the $S(u_jw_j)$ are Cauchy in $L^\gamma(v,\Omega)$ and so converge.  Denote their limit by $S(uw)$. 

Finally, to show  that $S(uw)=uS(w)+wS(u)$ in $L^\gamma(v,\Omega)$, we will show that  $u_iS(w_i)\to uS(w)$ and  $w_iS(u_i)\to wS(u)$  in $L^\gamma(v,\Omega)$.  The same argument holds for each, and so we will prove the first limit. If we repeat the argument above in~\eqref{eqn:zero-1.1}, adding and subtracting $uS(w_i)$, we have that
\begin{align*}
& \|u_iS(w_i) - uS(w)\|_{L^\gamma(v,\Omega)} \\
 & \quad \leq \|(u_i-u)S(w_i)\|_{L^\gamma(v,\Omega)} + \|uS(w_i-w)\|_{L^\gamma(v,\Omega)}   \\  
    &  \quad\leq \|u_i-u\|^\gamma_{L^{2\sigma}(v,\Omega)} \,\|S(w_i)\|^\gamma_{L^2(v,\Omega)}
    + \|u\|^\gamma_{L^{2\sigma}(v,\Omega)} \,\|S(w_i-w)\|^\gamma_{L^2(v,\Omega)} \\
    & \quad\leq S(2,\sigma)^\gamma C(\bS)^\gamma\big[ \|u_i-u\|^\gamma_{QH^{1,2}(v, \Omega)} \|w_i\|^{\gamma}_{QH^{1,2}(v, \Omega)}
    + \|u\|^\gamma_{QH^{1,2}(v, \Omega)} \|w_i-w\|^{\gamma}_{QH^{1,2}(v, \Omega)} \big].  
\end{align*}
The last term goes to $0$ as $i\to \infty$; this gives us the desired limit.

\medskip 

Now suppose that $(2)$ holds.  The proof is very similar to the previous proof, except that we need to use H\"older's  inequality in the scale of Orlicz spaces.  We will describe the relevant changes. Arguing as we did before, we have that 
\begin{align*}
\|S(u_iw_i) - S(u_jw_j)\|_{L^B(v,\Omega)} & \leq \|u_iS(w_i - w_j)\|_{L^B(v,\Omega)} + \|w_iS(u_i - u_j)\|_{L^B(v,\Omega)} \\
        &\qquad  + \|(u_i-u_j)S(w_j)\|_{L^B(v,\Omega)} + \|(w_i-w_j)S(u_j)\|_{L^B(v,\Omega)} \\
        & = M_1+M_2+M_3+M_4.
\end{align*}
The estimates for each of the four terms are essentially the same, so we only show that for $M_1$.   Since $t^\frac{1}{2}A^{-1}(t) \lesssim B^{-1}(t)$, by \eqref{eqn:gen-orlicz-holder}, \eqref{eqn:norm-subunit} and the Sobolev inequality \eqref{eqn:global-sobolev-orlicz}, we have that
\begin{multline*}
M_1 \leq K\|u_i\|_{L^A(v,\Omega)}\|S(w_i-w_j)\|_{L^2(v,\Omega)} \\ 
\leq KS(2,A)C(\bS)\|u_i\|_{QH^{1,2}(v,\Omega)}\|w_i-w_j\|_{QH^{1,2}(v,\Omega)}.
\end{multline*}
We have similar estimates for $M_2$, $M_3$, and $M_4$, so we get that the  $S(u_iw_i)$ are Cauchy in $L^B(v,\Omega)$.  Denote their limit in $L^B(v,\Omega)$  by $S(uw)$.  To show that $S(uw) = uSw+wSu$ as $L^B(v,\Omega)$ functions we need to show that 
$u_iS(w_i)\to uS(w)$ and  $w_iS(u_i)\to wS(u)$ in $L^B(v,\Omega)$.  We only show the first limit.  If we essentially repeat the estimate for $M_1$, we get 
\begin{align*}
& \|u_iSw_i - uSw\|_{L^B(v,\Omega)} \\
& \qquad \quad \leq \|u_iS(w_i-w)\|_{L^B(v,\Omega)}
+ \|(u_i-u)Sw\|_{L^B(v,\Omega)}\\
&\qquad \quad \leq C\big[\|u_i\|_{L^A(v,\Omega)}\|S(w_i-w)\|_{L^2(v,\Omega)}+\|u_i-u\|_{L^A(v,\Omega)}\|Sw\|_{L^2(v,\Omega)}\big]\\
&\qquad \quad \leq CS(2,A)\big[\|u_i\|_{QH^{1,2}(v,\Omega)}\|w_i-w\|_{QH^{1,2}(v,\Omega)}
+ \|u_i-u\|_{QH^{1,2}(v,\Omega)}\|w\|_{QH^{1,2}(v,\Omega}\big].
\end{align*}
The desired limit follows immediately. 

\medskip

Finally, suppose $(3)$ holds.  Then using the Sobolev inequality \eqref{eqn:global-sobolev} with $\sigma = 1$, and taking $\gamma=1$, we can repeat the entire argument for $(1)$ to show that $S(uw)$ is defined as an $L^1(v,\Omega)$ function and that $S(uw)=uS(w)+wS(u)$.  This completes the proof. 
\end{proof}

\medskip

We now use Lemma~\ref{lemma:subunicity-cond} to prove that we can use functions in $QH^{1,2}_0 (v, \Omega)$ as test functions.

\begin{lemma} \label{lemma:QH1-test-func}
Suppose $v\in L^1(\Omega)$.  Given Hypothesis~\ref{hyp:global-hyp}, let  $u \in QH^{1,2}_0 (v, \Omega)$ be a degenerate weak solution of 
the Dirichlet problem~\eqref{eqn:dirichlet}, that is, \eqref{eqn:weak-soln} holds for all $\varphi \in Q\lip_0(v,\Omega)$.  Then \eqref{eqn:weak-soln} holds for all $\varphi \in QH^{1,2}_0(v,\Omega)$ if one of the following is true:
\begin{enumerate}
    \item The degenerate Sobolev inequality~\eqref{eqn:global-sobolev} holds for some $\sigma>1$, and for some $q>2\sigma'$, $\bG,\,\bH \in L^q(v,\Omega,\R^N)$, $F\in L^{\frac{q}{2}}(v,\Omega)$.

     \item The degenerate Sobolev inequality~\eqref{eqn:global-sobolev-orlicz} holds with $A(t)=t^2\log(e+t)^{\sigma}$ for some $\sigma>0$ and  $\bG,\,\bH \in L^C(v,\Omega,\R^N)$, $F\in L^{D}(v,\Omega)$, where $C(t)=\exp\big(t^{\frac{2}{\lambda\sigma}}\big)-1$ and $D(t)=\exp\big(t^{\frac{1}{\lambda\sigma}}\big)-1$ for some $0<\lambda<1$. 

      \item The degenerate Sobolev inequality~\eqref{eqn:global-sobolev} holds without gain (i.e., $\sigma=1$) and  $\bG,\,\bH \in L^\infty(v,\Omega,\R^N)$, $F\in L^{\infty}(v,\Omega)$.
\end{enumerate}
\end{lemma}

\begin{remark}
    As we alluded in Remark~\ref{rem:subunit-defn},  the proof of Lemma~\ref{lemma:QH1-test-func} also shows that with these hypotheses, every integral in~\eqref{eqn:weak-soln} is finite for any $\varphi \in Q\lip_\loc(v,\Omega)$.  
    \end{remark}
    
    \begin{remark} We make the assumption that $v\in L^1(\Omega)$ here because it allows us to work more easily with the Young functions in case (2).  
\end{remark}

\begin{proof}
 Fix $\varphi \in QH_0^{1,2}(v,\Omega)$ and let $\{\varphi_k\}_{k=1}^\infty$ be a sequence of functions in $Q\lip_\loc(v,\Omega)$ that converge to $\varphi$ in $QH^{1,2}(v,\Omega)$ norm.  We need to show that with our hypotheses, each of the integrals in~\eqref{eqn:weak-soln} is the limit as $k\to \infty$ of the same integral with $\varphi$ replaced by $\varphi_k$.  The proofs are all variations of the same idea, so we will only prove that 
 \begin{equation} \label{eqn:test-func1}
 \lim_{k\to\infty} \int_\Omega u\bG \cdot \bS\varphi_k\, \,vdx = \int_\Omega u\bG \cdot \bS\varphi\, \,vdx. 
 \end{equation}
 By H\"older's inequality and Lemma~\ref{lemma:subunicity-cond},
 \begin{multline} \label{eqn:test-func2}
 \bigg|\int_\Omega u\bG \cdot \bS\varphi_k\, \,vdx - \int_\Omega u\bG \cdot \bS\varphi\, \,vdx\bigg|
 \leq \int_\Omega |u\bG||\bS(\varphi_k-\varphi)|\, \,vdx \\
 \leq \|u\bG\|_{L^2(v,\Omega)}\|\bS(\varphi_k-\varphi)\|_{L^2(v,\Omega)} 
 \leq C(\bS)\|u\bG\|_{L^2(v,\Omega)}\|\varphi_k-\varphi\|_{QH^{1,2}(v,\Omega)}.  
\end{multline}

Suppose now $(1)$ holds.  Then by H\"older's inequality with exponent $\sigma$ and the Sobolev inequality (which also implies that there is a Sobolev inequality without gain), the last term above is bounded by 
\begin{multline*} C(\bS)(1+S(2,1)) \|u\|_{L^{2\sigma}(v,\Omega)}^2 \|\bG\|_{L^{2\sigma'}(v,\Omega)}^2\|\varphi_k-\varphi\|_{QL^2(\Omega)} \\
\leq CS(2,\sigma) v(\Omega)^{\frac{4\sigma'}{q}}\|\nabla u\|_{QL^2(\Omega)}^2\|\bG\|_{L^{q}(v,\Omega)}^2 \|\varphi_k-\varphi\|_{QL^2(\Omega)}.
\end{multline*}
By our assumptions the first three terms are finite and the last term tends to $0$ as $k\to\infty$.  This shows that the limit~\eqref{eqn:test-func1} holds.

If $(2)$ holds, then starting from~\eqref{eqn:test-func2} we can argue as before; we just need to prove that $ \|u\bG\|_{L^2(v,\Omega)}$ is finite.
By H\"older's inequality in the scale of Orlicz spaces with $\Phi(t)=\exp(t^{\frac{1}{\sigma}})-1$ and $\bar{\Phi}(t)\approx t\log(e+t)^\sigma$, and by rescaling,
\begin{multline*} \|u\bG\|_{L^2(v,\Omega)} = \|u^2\bG^2\|_{L^1(v,\Omega)}^{\frac{1}{2}} \\
\leq 2\|u^2 \|_{L^{\bar{\Phi}}(v,\Omega)^{\frac{1}{2}}}\|\bG^2\|_{L^\Phi(v,\Omega)}^{\frac12}
\leq C\|u\|_{L^A(v,\Omega)}\|\bG\|_{L^C(v,\Omega)} < \infty. 
\end{multline*}

Finally, if $(3)$ holds, we can argue in exactly the same way, since we have that
\[ \|u\bG\|_{L^2(v,\Omega)} \leq \|\bG\|_{L^\infty(v,\Omega)} \|u\|_{L^2(v,\Omega)} < \infty. \]
This completes the proof.
\end{proof}

Finally, we prove that the compatibility condition~\eqref{eqn:compatible} holds for a larger collection of test functions.

\begin{lemma} \label{lemma:QH1-compatibility}
    Suppose $v\in L^1(\Omega)$ and $|Q(x)|_\op \leq kv(x)$ a.e. Given Hypothesis~\ref{hyp:global-hyp}, let  $(u,\nabla u),\, (w, \nabla w) \in QH^{1,2}_0 (v, \Omega)$ be such that $u,\, w \geq 0$.  If the conditions (1), (2), or (3) in Lemma~\ref{lemma:QH1-test-func} hold, then 
    \begin{equation} \label{eqn:qh1-compatibility} 
    \int_\Omega (\bG \cdot \bS(uw) + Fuw) \,vdx \geq 0.
    \end{equation}
\end{lemma}

\begin{proof}
    The proof is similar to the ideas used in the proofs of Lemmas~\ref{product-rule-all-sobolevs} and~\ref{lemma:QH1-test-func}, and we will refer back for some details. Fix  $u,\,w \in  QH^{1,2}_0(v, \Omega)$, $u,\,w \geq 0$.  By Lemma~\ref{lemma:non-neg-test}, there exist sequences $\{u_j\}_{j=1}^\infty$ and $\{w_j\}_{j=1}^\infty$ of non-negative functions in  $Q\lip_0(v,\Omega)$ that converge to $u$ and $w$ in $QH^{1,2}_0(v, \Omega)$.  Then for each $j$, $u_jw_j \in Q\lip_0(v,\Omega)$ and $u_jw_j \geq 0$, and so by the compatibility condition \eqref{eqn:compatible},
    \[  \int_\Omega (\bG \cdot \bS(u_jw_j) + Fu_jw_j) \,vdx \geq 0. \]
    Therefore, to prove that~\eqref{eqn:qh1-compatibility} holds, it will suffice to prove that
    \[ \lim_{j\rightarrow \infty}  \bigg( \int_\Omega \bG \cdot \bS(u_jw_j)\,vdx  
    + \int_\Omega Fu_jw_j \,vdx \bigg)
    =  \int_\Omega \bG \cdot \bS(uw)\,vdx + \int_\Omega Fuw \,vdx.\]
    We will prove that the limit of the first integral holds; the proof for the second is similar and somewhat simpler.  

    By Lemma~\ref{product-rule-all-sobolevs} we have that 
    \[ S(uw) = wSu + uSw, \qquad S(u_jw_j) = w_jSu_j + u_jSw_j.\]
    If we apply these identities and add and subtract the terms $w_jSu$ and $u_jSw$, we get that
    \begin{align*}
      &  \bigg| \int_\Omega \bG\cdot \bS(uw) - \bG\cdot \bS(u_jw_j) \,vdx \bigg| \\
      &  \qquad \leq \int_\Omega |\bG||w_j||S(u-u_j)|\,vdx 
       + \int_\Omega |\bG||w- w_j||S(u)|\,vdx \\
      & \qquad \qquad  + \int_\Omega |\bG||u_j||S(w-w_j)|\,vdx 
       + \int_\Omega |\bG||u- u_j||S(w)|\,vdx.
    \end{align*}
    We now argue as we did in the proof of~\eqref{eqn:test-func2} on each of these four terms:  we first apply H\"older's inequality and Lemma~\ref{lemma:subunicity-cond}, and then argue using the hypotheses in case (1), (2), (3).   In each case we get that the integral goes to $0$ as $j\rightarrow 0$ since $u_j\to u$ and $w_j\to w$ in $QH_0^{1,2}(v,\Omega)$.   This completes the proof.
\end{proof}

\section{Uniqueness of solutions of the zero problem}
\label{section:zero}

In this section we show that the zero function, ${\bf u}=(0,0)\in QH^{1,2}_0(v,\Omega)$, which is a degenerate weak solution of the  homogeneous Dirichlet problem (or simply, the zero problem),
\begin{equation}\label{homogeneous problem}
         \begin{cases}
 Lu = 0, & x \in \Omega, \\
 u = 0, & x \in \partial \Omega,
 \end{cases}
    \end{equation}
    is unique, given the hypotheses in Theorems~\ref{thm:main-theorem},~\ref{thm:log-gain}, and~\ref{thm:no-gain}. In particular, we will prove uniqueness assuming that a Sobolev inequality holds with gain $\sigma>1$, with Orlicz gain, and with no gain ($\sigma=1$).  
The uniqueness of the zero solution will be used in the proofs of these results when we apply the Fredholm alternative in the Hilbert space $QH_0^{1,2}(v,\Omega)$. 

\begin{theorem}\label{thm:hom-problem}
Suppose $v\in L^1(\Omega)$ and $|Q(x)|_\op \leq kv(x)$. a.e. Given  Hypothesis~\ref{hyp:global-hyp}, ${\bf u}=(0,0)$ in $QH_0^{1,2}(v,\Omega)$ is the unique solution of the homogeneous Dirichlet problem \eqref{homogeneous problem} if one of the following is true:
\begin{enumerate}
    \item The degenerate Sobolev inequality~\eqref{eqn:global-sobolev} holds for some $\sigma>1$, and for some $q>2\sigma'$, $\bG,\,\bH \in L^q(v,\Omega,\R^N)$, $F\in L^{\frac{q}{2}}(v,\Omega)$.

     \item The degenerate Sobolev inequality~\eqref{eqn:global-sobolev-orlicz} holds with $A(t)=t^2\log(e+t)^{\sigma}$ for some $\sigma>0$, and  $\bG,\,\bH \in L^C(v,\Omega,\R^N)$ and $F\in L^{D}(v,\Omega)$, where $C(t)=\exp\big(t^{\frac{2}{\lambda\sigma}}\big)-1$ and $D(t)=\exp\big(t^{\frac{1}{\lambda\sigma}}\big)-1$ for some $0<\lambda<1$. 

      \item The degenerate Sobolev inequality~\eqref{eqn:global-sobolev} holds without gain (i.e., $\sigma=1$),  $F\in L^{\infty}(v,\Omega)$,  $\bG,\,\bH \in L^\infty(v,\Omega,\R^N)$, and
\begin{equation} \label{eqn:perturbation}
 \max(C(\bR),C(\bS))S(2,1)\,\big( \|\bG\|_{L^\infty(v,\Omega)} + \|\bH\|_{L^\infty(v,\Omega)}\big)<1.
\end{equation}

\end{enumerate}
\end{theorem}

\begin{remark} 
While case $(3)$ of Theorem \ref{thm:hom-problem} is true for small coefficients $\bH,\bG$, the case where these functions are zero plays an important role in  applications. To emphasize this we note specifically that, under the hypotheses of each case of Theorem~\ref{thm:hom-problem},  the solution ${\bf u}=(0,0)$ of \eqref{homogeneous problem} is unique when $\bH=\bG=0$.
\end{remark}

\begin{proof}[Proof of Theorems~\ref{thm:hom-problem}]
Let $\vecu=(u,\nabla u)\in  QH_0^{1,2}(v,\Omega)$ be a weak solution of the zero problem \eqref{homogeneous problem}.
Recall that $u$ and $\nabla u$ are defined as the limit of a sequence of Lipschitz functions and their gradients in $L^2(v,\Omega)\oplus QL^2(\Omega)$; thus, while $u$ and $\nabla u$ are unique as limits, we do not have that $\nabla u$ is  uniquely determined by $u$.  Consequently, we must prove separately that  $u=0$ in $L^2(v,\Omega)$ and $\nabla u=0$ in $QL^2(\Omega)$.   However, suppose for the moment that we can show that if $\vecu \in QH_0^{1,2}(v,\Omega)$ is a weak solution, then $u=0$, that is, $\vecu = (0,\vecf)$ for some $\vecf \in QL^2(\Omega)$.  Then by the definition of a weak solution \eqref{eqn:weak-soln} and Lemma~\ref{lemma:QH1-test-func} we have that for any $\vecw = (w,\nabla w) \in QH_0^{1,2}(v,\Omega)$,
\[ 0 = \int_\Omega Q\vecf \cdot \nabla w \,dx + \int_\Omega \bH \cdot \bR \vecf w \, vdx.   \] 
In particular, we can take $\vecw=\vecu$, in which case the second term is $0$ and we get
\[ \int_\Omega |\sqrt{Q}\vecf|^2\,dx  = 0.  \]
Hence, we have that $\vecf=0$ in $QL^2(\Omega)$, which would complete the proof.

\medskip

Therefore, we need to prove that if $\vecu=(u,\nabla u)\in  QH_0^{1,2}(v,\Omega)$ is a weak solution, then $u=0$ in $L^2(v,\Omega)$.  To do so, we will first argue in general, and then treat each case (1), (2), and (3) in the hypotheses separately.  We argue by contradiction.  Suppose to the contrary that there exists a solution $\vecu$ such that  $u\neq 0$. Suppose $\sup_\Omega u>0$.  Note that in this case, by the Sobolev inequality (which holds in each case), $\|\nabla u\|_{QL^2(\Omega)}>0$.  

Fix  $k>0$ such that $0<k<\sup_\Omega u$.  By assumption we have that $v\in L^1(\Omega)$ and $|Q|_\op\leq kv$, and so,  by~Lemma~\ref{lemma:truncate}, $(w,\nabla w)=((u-k)_+,\chi_{\{x: u(x)>k\}}\nabla u) \in QH_{0}^{1,2}(v,\Omega)$.  Therefore, by Lemma~\ref{lemma:QH1-test-func}, we can use it as a test function for $u$ in \eqref{eqn:weak-soln}.  In other words, we have that
\[ \int_\Omega Q\nabla u\cdot \nabla w\, dx + \int_\Omega \bH\cdot\bR u w \, vdx
+ \int_\Omega u\bG \cdot\bS w\, vdx+\int_\Omega Fuw\, vdx=0. \]
In each of cases (1)--(3), by Lemma~\ref{product-rule-all-sobolevs} we have that $\bG\cdot\bS(uw)=u\bG\cdot\bS w+w\bG\cdot\bS u$.  Thus the above equality becomes
\begin{equation*}
\int_\Omega Q\nabla u \cdot \nabla w\, dx
=\int_\Omega w \big(\bG\cdot\bS u -\bH\cdot\bR u\big)\, vdx
- \int_\Omega \big(\bG\cdot\bS (uw) +Fuw\big)\, vdx.
\end{equation*}

We can now estimate as follows. First note that if we let $u^+=\max\{u,0\}$, then, again by Lemma~\ref{lemma:truncate}, $(u^+,\nabla u^+) \in QH_0^{1,2}(v,\Omega)$, and $uw=u^+w$.  Therefore, if we apply Lemma~\ref{lemma:QH1-compatibility} to the functions $(u^+,\nabla u^+)$ and $(w,\nabla w)$, in every case we have that the last integral is nonnegative, and so 
\[ \int_\Omega Q\nabla u \cdot \nabla w\, dx
\leq \int_\Omega w \big(\bG\cdot\bS u -\bH\cdot\bR u\big)\, vdx. \]
Let $\Gamma=\supp(w)$; note that $\supp(\nabla w)\subset \Gamma$ and $\nabla u = \nabla w$ on $\Gamma$.  Therefore, by H\"older's inequality and Lemma~\ref{lemma:subunicity-cond},
\begin{align}
 \qquad   & \int_\Gamma Q\nabla w \cdot \nabla w\, dx
    \leq\int_\Gamma w(\bG\cdot\bS-\bH\cdot\bR) w\, vdx \notag\\
    &  \qquad \leq\int_\Gamma|w|\big(|\bG||\bS w|+|\bH||\bR w|\big)\,vdx; \notag \\
    &  \qquad\leq  \|w|\bG|\|_{L^2(v,\Gamma)}\|\bS w\|_{L^2(v,\Gamma)} + \|w|\bH|\|_{L^2(v,\Gamma)}\|\bR w\|_{L^2(v,\Gamma)} \notag \\
    &  \qquad\leq  \bigg(C(\bS)\|w|\bG|\|_{L^2(v,\Gamma)}+C(\bR)\|w|\bH|\|_{L^2(v,\Gamma)} \bigg)\|w\|_{QH^{1,2}(v,\Gamma)}; \notag
    \intertext{by the Sobolev inequality with no gain (which holds in every case),}
    &  \qquad\leq C\big( \bR, \bS, S(2,1)\big)\bigg(\|w|\bG|\|_{L^2(v,\Gamma)}+\|w|\bH|\|_{L^2(v,\Gamma)}\bigg)\|\nabla w\|_{QL^2(v,\Gamma)}.\label{eq:3.5}
    \end{align}

\medskip

To complete the proof, we treat three cases.  Suppose first that $(1)$ holds.  
We estimate the first term on the right-hand side of~\eqref{eq:3.5}.  By  H\"older's inequality with exponents $\frac{q}{2},\frac{q}{q-2}$ we have that
\begin{align*}
    \bigg(\|w|\bG|\|_{L^2(v,\Gamma)}+\|w|\bH|\|_{L^2(v,\Gamma)}\bigg)
    &\leq(\|\bG\|_{L^q(v,\Gamma)}+\|\bH\|_{L^q(v,\Gamma)}) \|w\|_{L^{\frac{2q}{q-2}}(v,\Gamma)}\\
    &=C(\bG,\bH)\|w\|_{L^{\frac{2q}{q-2}}(v,\Gamma)}. \\
    \intertext{Since $q>2\sigma'$, by H\"older's inequality  with exponents
$p=\frac{\sigma(q-2)}{q}>1$ and $p'=\frac{\sigma(q-2)}{\sigma(q-2)-q}$, }
    &\leq C(\bG,\bH)v\left(\Gamma\right)^{\frac{1}{2\sigma'}-\frac{1}{q}} \|w\|_{L^{2\sigma}(v,\Gamma)}. \\
\intertext{By the Sobolev inequality \eqref{eqn:global-sobolev} with gain $\sigma>1$,}
   &\leq C(\bG,\bH)S(2,\sigma)v\left(\Gamma\right)^{\frac{1}{2\sigma'}-\frac{1}{q}} \|\nabla w\|_{QL^2(\Gamma)}.
\end{align*}
If we insert this  estimate into \eqref{eq:3.5}, we get
\begin{equation*}
    \int_\Gamma |\sqrt{Q}\nabla w|^2~dx\leq C v\left(\Gamma\right)^{\frac{1}{2\sigma'}-\frac{1}{q}}\|\nabla w\|^2_{QL^2(\Gamma)},
\end{equation*}
where $C=C({\bf H},{\bf G},{\bf R},{\bf S},S(2,1),S(2,\sigma))>0$ does not depend on $k$. 
If  $\|\nabla w\|^2_{QL^2(\Gamma)}>0$, we can divide both sides by $\|\nabla w\|^2_{QL^2(\Gamma)}$ and rearrange terms; then, since $\frac{1}{2\sigma'}-\frac{1}{q}>0$, we get
     \begin{equation*}
         v\left(\Gamma\right)\geq C^{-\frac{2q\sigma'}{q-2\sigma'}}>0. 
         \end{equation*}
In other words, $\supp(w)$ 
has positive $v$-measure independent of $k$.  But, as  $k\to \sup_\Omega u$, since $u\in L^2(v,\Omega)$, we must have that  $w\rightarrow 0$ $v$-a.e. in $\Omega$, and so $v(\Gamma)\to 0$ since $v\in L^1(\Omega)$.  Thus, we have a contradiction, and so we must have that $\sup_\Omega u \leq 0$.   If we now assume that $\inf_\Omega u <0$, then we can repeat the above argument with  $k>0$ such that $0<k<-\inf_\Omega u$ and $w=((u+k)_-,\chi_{\{x: u(x)<-k\}}\nabla u)$.  This again yields a contradiction, and so we must have $\inf_\Omega u\geq 0$.  
 Therefore, $u=0$ $v$-a.e. in $\Omega$.

\medskip

Now suppose that $(2)$ holds.  The proof is very similar, but we use H\"older's inequality in the scale of Orlicz spaces.   Let $\Phi(t) = C(t^{\frac{1}{2}}) = \exp(t^\frac{1}{\lambda\sigma})-1$; then $\bar{\Phi}(t)\approx t\log(e+t)^{\lambda \sigma}$.  Then by  inequality \eqref{eqn:orlicz-holder} and rescaling, we get that
\begin{align*}
    \|w|\bG|\|_{L^2(v,\Gamma)}
    & \leq \big(2\|w^2\|_{L^{\bar{\Phi}}(v,\Gamma)}\||\bG|^2\|_{L^\Phi(v,\Gamma)}\big)^\frac{1}{2} \\
    & = \sqrt{2}\|w\|_{L^{\bar{\Phi}_2}(v,\Gamma)}\||\bG|\|_{L^{\Phi_2}(v,\Gamma)}. \\
    \intertext{Let $\Psi(t)=\exp(t^\frac{2}{\sigma(1-\lambda)})-1$; since $A^{-1}(t)\Psi^{-1}(t) \lesssim \bar{\Phi}^{-1}(t)$, by the generalized Orlicz Hölder inequality \eqref{eqn:gen-orlicz-holder} and the Sobolev inequality \eqref{eqn:global-sobolev-orlicz},}
    & \leq C(\bG) K\|w\|_{L^{A}(v,\Gamma)}\|\chi_\Gamma\|_{L^{\Psi}(v,\Gamma)} \\
    & \leq C(\bG,K)S(2,A)\|\nabla w\|_{QL^{2}(v,\Gamma)}\|\chi_\Gamma\|_{L^{\Psi}(v,\Gamma)}.
\end{align*}
The same estimate also holds with $\bG$ replaced by $\bH$.  

If we now insert these estimates into~\eqref{eq:3.5}, we get
\begin{equation*}
      \int_\Omega Q\nabla w \cdot \nabla w\, dx\leq C\|\chi_\Gamma\|_{L^\Psi(v,\Gamma)}\|\nabla w\|^2_{{QL^2}(\Gamma)},
\end{equation*}
where $C=C(\bH,\bG,\bR,\bS,S(2,1),S(2,A))\geq 1$.  If $ \|\nabla w\|^2_{{QL^2}(\Gamma)}>0$, we can rearrange terms to get
\begin{equation*}
     C\|\chi_\Gamma\|_{L^\Psi(v,\Gamma)} \geq 1.
\end{equation*}
If we use the formula $\|\chi_\Gamma\|_{L^\Psi(v,\Gamma)}=[\Psi^{-1}(v(\Gamma)^{-1})]^{-1}$ and solve for $v(\Gamma)$, we get
\begin{equation*}
    v(\Gamma) \ge \big(\exp(C^{\frac{2}{\sigma(1-\lambda)}}\big)^{-1}  > 0.
\end{equation*}
We can now argue exactly as we did at the end of previous case, and again conclude that $u=0$ $v$-a.e. in $\Omega$.

\medskip 

Finally, suppose that $(3)$ holds.  This case is much simpler.  We repeat the argument that produced~\eqref{eq:3.5} using $u$ as a test function in place of $w$, and keeping the domain of integration $\Omega$.  We then get, after keeping careful track of the constants, that
\begin{align*}
& \int_\Omega \nabla Qu \cdot \nabla u \, dx \\
&\qquad \qquad  \leq  \max(C(\bR),C(\bS))S(2,1)) \big( \|u|\bG| \|_{L^2(v,\Omega)} + \|u|\bH| \|_{L^2(v,\Omega)}\big)
\|\nabla u\|_{L^2{(v,\Gamma)}} \\
& \qquad \qquad \leq  \max(C(\bR),C(\bS))S(2,1))\big( \|\bG\|_{L^\infty(v,\Omega)} + \|\bH\|_{L^\infty(v,\Omega)}\big)
\|\nabla u\|_{L^2{(v,\Gamma)}}.
\end{align*}
If $\|\nabla u\|_{L^2{(v,\Gamma)}}>0$, we can rearrange terms to get
\[  \max(C(\bR),C(\bS))S(2,1))\big( \|\bG\|_{L^\infty(v,\Omega)} + \|\bH\|_{L^\infty(v,\Omega)}\big)\geq 1, \]
which contradicts our assumption that~ \eqref{eqn:perturbation} holds.  We again conclude that $u=0$ $v$-a.e. in $\Omega$.  This completes the proof.
\end{proof}

\section{Boundedness and coercivity:  Sobolev inequality with  power gain}
\label{section:coerce-power-gain}

In order to prove the existence and uniqueness of solutions to our Dirichlet problem, we are going to apply techniques from functional analysis.  In order to do so, we will consider the properties of the bilinear form defined by the left-hand side of~\eqref{eqn:weak-soln}.  More precisely, define  $\mathcal{L}:QH_0^{1,2}(v,\Omega)\times QH_0^{1,2}(v,\Omega)\to \R$ by
\begin{equation} \label{bilinear}
    \mathcal{L}(u,w)=\int_{\Omega} \nabla Qw\cdot \nabla u\, dx +\int_{\Omega}\bH\cdot\bR u\,w \, vdx+\int_{\Omega} u\bG \cdot \bS w\, v dx+\int_{\Omega}Fuw\, vdx.
\end{equation}
We will need that the bilinear form $\mathcal{L}$ is bounded and almost coercive.  In this section we will prove this assuming that there is a Sobolev inequality with power gain.  In Section~\ref{section:coerce-weaker-gain} we will prove this assuming a Sobolev inequality with Orlicz gain or with no gain.

\begin{lemma} \label{lemma:bounded-power-gain}
Given  Hypothesis~\ref{hyp:global-hyp}, suppose that the global degenerate Sobolev inequality~\eqref{eqn:global-sobolev} holds with gain $\sigma>1$.   If for some $q>2\sigma'$, $\bH,\,\bG \in L^{q}(v,\Omega,\R^N)$  and $F\in L^{\frac{q}{2}}(v,\Omega)$,  then the bilinear form $\mathcal{L}$ is bounded: that is, there is a $C>0$  such that for every $u,\,w \in QH_0^{1,2}(v,\Omega)$,
    \begin{equation*}
        |\mathcal{L}(u,w)|\leq C\|u\|_{QH^{1,2}(v,\Omega)}\cdot\|w\|_{QH^{1,2}(v,\Omega)}.
    \end{equation*}
\end{lemma}

\begin{remark}
    Though this result is only used for results with equations with lower order terms, we do not need the assumption that $v\in L^1(\Omega)$ and $|Q(x)|_\op \leq kv(x)$ a.e.  The same is true for Lemma~\ref{power-gain-almost-coercivity} below.
\end{remark}

\begin{proof}
    Fix $u,w \in QH_0^{1,2}(v,\Omega)$. Then by H\"older's inequality and  Lemma~\ref{lemma:subunicity-cond}, 
\begin{align}\label{first part of bil.bdd}
  |\mathcal{L}(u,w)|
   & \leq \int_\Omega|Q\nabla u\cdot \nabla w|\,dx
   +\int_\Omega|\bH||\bR u||w|\,v dx \notag \\
   & \qquad +\int_\Omega|u||\bG||\bS w|\,v dx
   +\int_\Omega|Fuw|\,v dx \notag \\
   &\leq \|\nabla u\|_{QL^2(\Omega)}\|\nabla w\|_{QL^2(\Omega)} 
  +  \|w|\bH|\|_{L^2(v,\Omega)} \|\bR u\|_{L^2(v,\Omega)}  \notag \\
   &\qquad + \| u|\bG| \|_{L^2(v,\Omega)} \|\bS w\|_{L^2(v,\Omega)} 
    + \| F^{\frac{1}{2}}u \|_{L^2(v,\Omega)}  \|F^{\frac{1}{2}}w \|_{L^2(v,\Omega)} \notag \\
   &\leq \| u\|_{QH^{1,2}(v,\Omega)}\| w\|_{QH^{1,2}(v,\Omega)} 
  +   C(\bR) \|w|\bH|\|_{L^2(v,\Omega)} \| u\|_{QH^{1,2}(v,\Omega)}  \notag \\
   &  \qquad+   C(\bS) \| u|\bG| \|_{L^2(v,\Omega)} \| w\|_{QH^{1,2}(v,\Omega)} 
    + \| F^{\frac{1}{2}}u \|_{L^2(v,\Omega)} \|F^{\frac{1}{2}}w \|_{L^2(v,\Omega)}. \\  
    \intertext{Now apply H\"older's inequality with exponents $\sigma$ and $\sigma'$ in the last three terms to get}
        &\leq\|u\|_{QH^{1,2}(v,\Omega)}\|w\|_{QH^{1,2}(v,\Omega)} \notag \\
        & \qquad +  C(\bR)\|\textbf{H}\|_{L^{2\sigma'}(v,\Omega)}\|w\|_{L^{2\sigma}(v,\Omega)}\|u\|_{QH^{1,2}(v,\Omega)} \notag \\
        &\quad\quad+ C(\bS)\|\textbf{G}\|_{L^{2\sigma'}(v,\Omega)}\|u\|_{L^{2\sigma}(v,\Omega)}\|w\|_{QH^{1,2}(v,\Omega)} \notag \\
        & \qquad +\|F\|_{L^{\sigma'}(v,\Omega)}\|u\|_{L^{2\sigma}(v,\Omega)}\|w\|_{L^{2\sigma}(v,\Omega)}. \notag 
        \intertext{Finally, by the Sobolev inequality \eqref{eqn:global-sobolev},}
        & \leq C\|u\|_{QH^{1,2}(v,\Omega)}\|w\|_{QH^{1,2}(v,\Omega)}\notag, 
\end{align}
where 
\[ C=1+ S(2,\sigma)\big( C(\bR)\|\textbf{H}\|_{L^{2\sigma'}(v,\Omega)}+
 C(\bS)\|\textbf{G}\|_{L^{2\sigma'}(v,\Omega)}
+S(2,\sigma)\|F\|_{L^{\sigma'}(v,\Omega)}\big). \]
\end{proof}

In our next result  we prove that the bilinear form $\mathcal L$ is almost coercive.  

\begin{lemma}\label{power-gain-almost-coercivity}
Given  Hypothesis~\ref{hyp:global-hyp}, suppose that the global degenerate Sobolev inequality~\eqref{eqn:global-sobolev} holds with gain $\sigma>1$. If for some $q>2\sigma'$, $\bH,\,\bG \in L^{q}(v,\Omega,\R^N)$  and $F\in L^{\frac{q}{2}}(v,\Omega)$,  then the bilinear form $\mathcal{L}$ is almost coercive: that is, there exist constants $c_1,\,C_3>0$ such that for every $u\in QH_0^{1,2}(v,\Omega)$, 
    \begin{equation}\label{almcoercivity}
        \mathcal{L}(u,u)\geq c_1\|u\|_{QH^{1,2}(v,\Omega)}^2-C_3\|u\|_{L^2(v,\Omega)}^2.
    \end{equation}
\end{lemma}
\begin{proof}
    Fix $u\in QH_0^{1,2}(v,\Omega)$; then we have that
    \begin{equation}\label{eq-coer}
    \begin{split}
        \mathcal{L}(u,u)&\geq \int_\Omega|\sqrt{Q}\nabla u|^2 dx
        -\bigg|\int_\Omega u\big(\bH\cdot\bR+\bG \cdot\bS\big) u\,vdx\bigg|
        -\bigg|\int_\Omega Fu^2\,vdx\bigg|\\
                &\geq \|\nabla u\|_{QL^2(\Omega)}^2-I_1-I_2.
    \end{split}
    \end{equation}
   We first estimate $I_1$.  By H\"older's inequality, Lemma~\ref{lemma:subunicity-cond}, and the Sobolev inequality with no gain (which holds in this case),
\begin{align}\label{eq:I1}
    I_1 
    & \leq \int_\Omega |u||\bH||\bR u| +|u\|\bG||\bS u|\,vdx  \\
    & \leq \|u |\bH|\|_{L^2(v,\Omega)}\|\bR u\|_{L^2(v,\Omega)} 
        + \|u |\bG|\|_{L^2(v,\Omega)}\|\bS u\|_{L^2(v,\Omega)} \notag \\
    & \leq  C(\bR)S(2,1) \|u |\bH|\|_{L^2(v,\Omega)}\| \nabla u\|_{QL^2(\Omega)} \notag \\
    &  \qquad \quad +  C(\bS) S(2,1)\|u |\bG|\|_{L^2(v,\Omega)}\|\nabla u\|_{QL^2(\Omega)}; \notag \\
        \intertext{if we now apply Hölder's inequality with the exponents $\frac{q}{2}$ and $(\frac{q}{2})'=\frac{q}{q-2}$, we get}
    & \leq C(\bR)S(2,1)\|\bH\|_{L^q(v,\Omega) }\|u^2\|_{L^\frac{q}{q-2}(v,\Omega)}^{\frac{1}{2}}\|\nabla u\|_{QL^2(\Omega)} \notag \\
    & \qquad \quad + C(\bS)S(2,1)\|\bG\|_{L^q(v,\Omega) }\|u^2\|_{L^\frac{q}{q-2}(v,\Omega)}^{\frac{1}{2}} \|\nabla u\|_{QL^2(\Omega)}\notag \\
    &  = C_0 \|u^2\|_{L^\frac{q}{q-2}(v,\Omega)}^{\frac{1}{2}}\|\nabla u\|_{QL^2(\Omega)}, \notag
\end{align}
where $C_0=C(\bR)S(2,1)\|\bH\|_{L^q(v,\Omega) }+C(\bS)S(2,1)\|\bG\|_{L^q(v,\Omega) }$.

Since $1<(\frac{q}{2})'<\sigma$, by the interpolation of the Lebesgue norms (see~\cite[inequality~(7.9)]{MR1814364}), Young's inequality~\eqref{eqn:youngs-ineq} with $p=\frac{1}{\lambda}$ and $\eta=\epsilon$ ($\epsilon>0$ to be fixed below), and by rescaling, we have that 
\begin{multline*}  \|u^2\|_{L^{\frac{q}{q-2}}(v,\Omega)} 
 =\|u^{2\lambda}u^{2(1-\lambda)}\|_{L^{\frac{q}{q-2}}(v,\Omega)}
 \leq \|u^2\|^\lambda_{L^\sigma(v,\Omega)}\|u^2\|^{1-\lambda}_{L^1(v,\Omega)} \\
 \leq \epsilon\|u^2\|_{L^\sigma(v,\Omega)}+\epsilon^{-\mu}\|u^2\|_{L{^1}(v,\Omega)}
        =\epsilon\|u\|^2_{L^{2\sigma}(v,\Omega)}+
        \epsilon^{-\mu}\|u\|^{2}_{L^{2}(v,\Omega)},
\end{multline*}
where $\mu=\frac{1-\lambda}{\lambda}$.        
If we combine this with the previous estimate and use the Sobolev inequality~\eqref{eqn:global-sobolev}, we get
\begin{align*}
        I_1&\leq C_0\big(\epsilon^{\frac{1}{2}}\|u\|_{L^{2\sigma}(v,\Omega)}+\epsilon^{\frac{-\mu}{2}}\|u\|_{L^2(v,\Omega)}\big) \|\nabla u\|_{QL^2(\Omega)}\\
    &=C_0S(2,\sigma)\epsilon^{\frac{1}{2}}\|\nabla u\|_{QL^{2}(\Omega)}^2
    + C_0 \epsilon^{\frac{-\mu}{2}}\|u\|_{L^2(v,\Omega)}\|\nabla u\|_{QL^2(\Omega)}.\\
    \intertext{If we apply Young's inequality~\eqref{eqn:youngs-ineq} with $p=2$ and $\eta=\epsilon^{\frac{-(\mu+1)}{2}}$, we get }
    &\leq C_0 (S(2,\sigma)+1)\epsilon^{\frac{1}{2}}  \|\nabla u\|^2_{QL^2(\Omega)}
    + C_0\epsilon^{\frac{-(2\mu+1)}{2}}\|u\|^2_{L^2(v,\Omega)}.
   \end{align*}
Therefore, if we let $\epsilon= (4C_0(S(2,\sigma)+1)^{-2}$ and $C_1=C_0\epsilon^{\frac{-(2\mu+1)}{2}}$ we get 
\begin{equation}\label{eq:A}
    I_1\leq\frac{1}{4}\|\nabla u\|_{QL^2(\Omega)}^2+ C_{1}\|u\|_{L^2(v,\Omega)}^2.
\end{equation}

\medskip

We now estimate $I_2$.  The argument is similar to the above but simpler.  We use H\"older's inequality with exponents  $\frac{q}{2}$ and $(\frac{q}{2})'=\frac{q}{q-2}$, the estimate for $\|u^2\|_{L^{\frac{q}{q-2}}(v,\Omega)}$, and the Sobolev inequality  to get 
\begin{align*}
      I_2 \leq 
      &\|F\|_{L^{\frac{q}{2}}(v,\Omega)}\|u^2\|_{L^{\frac{q}{q-2}}(v,\Omega)}.\\
      &\leq\|F\|_{L^\frac{q}{2}(v,\Omega)}\big(\epsilon\|u\|_{L^{2\sigma}(v,\Omega)}^2+
        \epsilon^{-\mu}\|u\|_{L^{2}(v,\Omega)}^2\big)\\
        &\leq \|F\|_{L^\frac{q}{2}(v,\Omega)}\bigg(\epsilon{S(2,\sigma)}^2\|\nabla u\|_{QL^2(\Omega)}^2+
        \epsilon^{-\mu}\|u\|_{L^{2}(v,\Omega)}^2\bigg).
\end{align*}
Therefore, if we now set $\epsilon=\big(4S(2,\sigma)^2\|F\|_{L^{\frac{q}{2}}(v,\Omega)}\big)^{-1}$ and $C_2=\epsilon^{-\mu}\|F\|_{L^{\frac{q}{2}}(v,\Omega)}$ we get
\begin{equation}\label{eq:B}
    I_2\leq\frac{1}{4}\|\nabla u\|_{QL^2(\Omega)}^2+C_{2}\|u\|_{L^2(v,\Omega)}^2.
\end{equation}

If we now insert \eqref{eq:A} and \eqref{eq:B} into \eqref{eq-coer}, and then apply the Sobolev inequality with no gain, we get
\begin{equation*}
        \mathcal{L}(u,u) \geq \frac{1}{2}\|\nabla u\|_{QL^2(\Omega)}^2 -(C_1+C_2)\|u\|_{L^2(v,\Omega)}^2 \geq c_1\|u\|^2_{QH^{1,2}(v,\Omega)}-C_3\|u\|^2_{L^2(v,\Omega)},
 \end{equation*}
where $c_1=\big(\sqrt{2}(1+S(2,1))\big)^{-2}$ and $C_3=C_1+C_2$.  This completes the proof
\end{proof}

\section{Boundedness and coercivity:  Sobolev inequality with Orlicz gain or no gain}
\label{section:coerce-weaker-gain}

In this section we prove that the bilinear form~\eqref{bilinear} is bounded and almost coercive assuming the existence of a Sobolev inequality with Orlicz gain or no gain. We first consider the case of Orlicz gain.

\begin{lemma}\label{orlicz-gain-boundedness}
Suppose $v\in L^1(\Omega)$.  Given  Hypothesis~\ref{hyp:global-hyp}, suppose that the global degenerate Sobolev inequality~\eqref{eqn:global-sobolev-orlicz} holds with $A(t)=t^2\log(e+t)^\sigma$, $\sigma>0$.  
If for some $0<\lambda<1$, 
$\bH,\,\bG \in L^{C}(v,\Omega,\R^N)$, and $F\in L^{D}(v,\Omega)$, where $C(t)=\exp{(t^\frac{2}{\lambda\sigma})}-1$ and $D(t)=\exp{(t^\frac{1}{\lambda\sigma})}-1$, then
 the bilinear form $\mathcal{L}$ is bounded in $QH_0^{1,2}(v,\Omega)$: that is, there exists $C>0$  such that for every $u,w \in QH_0^{1,2}(v,\Omega)$,
    \begin{equation*}
        |\mathcal{L}(u,w)|\leq C\|u\|_{QH^{1,2}(v,\Omega)}\cdot\|w\|_{QH^{1,2}(v,\Omega)}
    \end{equation*}
\end{lemma}

\begin{remark}
    In Lemma~\ref{orlicz-gain-boundedness} we assume that $v\in L^1(\Omega)$ only because it makes it easier to work with the Orlicz norms.  The same is true for Lemma~\ref{orlicz-gain-almost-coercivity} below.
\end{remark}

\begin{proof}
    Fix $u,\,w \in QH_0^{1,2}(v,\Omega)$. We first argue as we did in the proof of Lemma~\ref{lemma:bounded-power-gain} to get inequality~\eqref{first part of bil.bdd}.   To complete the proof we need to estimate the last three terms in this inequality.  To estimate the first of these, note that $C^{-1}(t)\Psi^{-1}(t)\lesssim t^{\frac{1}{2}}$, where $\Psi(t)=t^2\log(e+t)^{\lambda\sigma} \lesssim A(t)$.  Therefore, by the generalized H\"older's inequality~\eqref{eqn:gen-orlicz-holder} and the Sobolev inequality~\eqref{eqn:global-sobolev-orlicz} with gain $\Psi$, we get
    \begin{multline*} \|w |\bH| \|_{L^2(v, \Omega)} 
    \leq K \|\bH\|_{L^{C}(v,\Omega)} \|w\|_{L^{\Psi}(v,\Omega)}\\
    \leq KS(2,\Psi)\|\bH\|_{L^{C}(v,\Omega)} \|\nabla w\|_{QL^2(\Omega)}
    \leq KS(2,\Psi)\|\bH\|_{L^{C}(v,\Omega)} \|w\|_{QH^{1,2}(v,\Omega)}.
    \end{multline*}
Similarly, we have that 
\begin{equation*}
      \|u\bG\|_{L^2(v, \Omega)} 
      \leq K S(2, \Psi)\|\bG\|_{L^{C}(v,\Omega)}\|u\|_{QH^{1,2}(v,\Omega)}.
\end{equation*}
To estimate the final term, we argue exactly as we did for the first term; by rescaling we get
\begin{equation*} \|F^{\frac{1}{2}}u \|^2_{L^2(v, \Omega)} 
    \leq K \|F^{\frac{1}{2}}u\|_{L^{C}(v,\Omega)} \|u\|_{L^{\Psi}(v,\Omega)}\\
    \leq KS(2,\Psi)\|F\|_{L^{D}(v,\Omega)}^{\frac{1}{2}}  \|u\|_{QH^{1,2}(v,\Omega)}.
\end{equation*}
The same inequality holds with $u$ replaced by $w$.  If we now combine these three estimates with inequality~\eqref{first part of bil.bdd}, we get
\[ |\mathcal{L}(u,w)| \leq C\|u\|_{QH^{1,2}(v,\Omega)}\|w\|_{QH^{1,2}(v,\Omega)}, \]
where 
$$C=  KS(2,\Psi)\big( C(\bR)\|\bH\|_{L^C(v,\Omega)} + C(\bS)\|\bG\|_{L^C(v,\Omega)}
+ KS(2,\Psi)\|F\|_{L^D(v,\Omega)}\big).$$
\end{proof}

Next, we prove that with the same hypotheses the bilinear form is almost coercive.

\begin{lemma}\label{orlicz-gain-almost-coercivity}
Suppose that $v\in L^1(\Omega)$.  Given  Hypothesis~\ref{hyp:global-hyp}, suppose that the global degenerate Sobolev inequality~\eqref{eqn:global-sobolev-orlicz} holds with $A(t)=t^2\log(e+t)^\sigma$, $\sigma>0$.  
If for some $0<\lambda<1$, 
$\bH,\,\bG \in L^{C}(v,\Omega,\R^N)$, and $F\in L^{D}(v,\Omega)$, where $C(t)=\exp{(t^\frac{2}{\lambda\sigma})}-1$ and $D(t)=\exp{(t^\frac{1}{\lambda\sigma})}-1$, then
 the bilinear form $\mathcal{L}$ is almost coercive: that is, there exist constants $c_1,\, C_3>0$ such that for every $u \in QH^{1,2}_0(v,\Omega)$,
    \begin{equation*}
        \mathcal{L}(u,u) \geq c_1\|u\|^2_{QH^{1,2}(v,\Omega)} - C_3\|u\|^2_{L^2(v,\Omega)}.
    \end{equation*}
\end{lemma}

\begin{proof}
Fix $u \in QH_0^{1,2}(v,\Omega)$. We first argue as we did in the proof of Lemma~\ref{power-gain-almost-coercivity}, dividing the estimate into three parts as in~\eqref{eq-coer}.  We first estimate $I_1$ by arguing as in~\eqref{eq:I1}, but then applying H\"older's inequality in the scale of Orlicz spaces~\eqref{eqn:orlicz-holder} and rescaling to get
\begin{align*}
   I_1 & \leq \sqrt{2}C(\bR)S(2,1)\||\bH|^2\|_{L^D(v,\Omega)}^{\frac{1}{2}} \|u^2\|_{L^{\bar{D}}(v,\Omega)}^{\frac{1}{2}} \|\nabla u \|_{QL^2(\Omega)} \\
   & \qquad \qquad + \sqrt{2}C(\bS)S(2,1)\||\bG|^2\|_{L^D(v,\Omega)}^{\frac{1}{2}} \|u^2\|_{L^{\bar{D}}(v,\Omega)}^{\frac{1}{2}} \|\nabla u \|_{QL^2(\Omega)} \\
      & \leq \sqrt{2}S(2,1)(C(\bR) \|\bH\|_{L^C(v,\Omega)}  + C(\bS) \|\bG\|_{L^C(v,\Omega)})
      \|u^2\|_{L^{\bar{D}}(v,\Omega)}^{\frac{1}{2}} \|\nabla u \|_{QL^2(\Omega)} \\
      & = C_0\|u^2\|_{L^{\bar{D}}(v,\Omega)}^{\frac{1}{2}} \|\nabla u \|_{QL^2(\Omega)}.
\end{align*}

If we let $\Psi(t)=t\log(e+t)^\sigma$, then $t^{1-\lambda}\Psi^{-1}(t)^\lambda \lesssim \bar{D}(t)^{-1}$, and so by Theorem~\ref{scottinterp2} we have that 
\begin{equation}\label{Orlicz interpolation}
            \|u^2\|_{L^{\Bar{\Phi}}(v,\Omega)}
           =\epsilon\|u\|_{L^A(v,\Omega)}^2+K^{\mu+1}
        \epsilon^{-\mu}\|u\|_{L^2(v,\Omega)}^2
\end{equation}
 where $\mu=\frac{1-\lambda}{\lambda}$ and  $\epsilon >0$ will be fixed below.  
 Therefore, if we combine these two inequalities and use the Sobolev inequality~\eqref{eqn:global-sobolev-orlicz}, we get
\begin{align*}
 I_1 
 & \leq C_0\big(\epsilon^\frac{1}{2}\|u\|_{L^A(v,\Omega)}+K^{\frac{\mu+1}{2}}\epsilon^{\frac{-\mu}{2}}\|u\|_{L^2(v,\Omega)}\big) 
 \|\nabla u \|_{QL^2(\Omega)} \\
 & \leq C_0S(2,A)\epsilon^\frac{1}{2}\|\nabla u \|_{QL^2(\Omega)}^2
 + C_0K^{\frac{\mu+1}{2}}\epsilon^{\frac{-\mu}{2}}\|u\|_{L^2(v,\Omega)}\|\nabla u \|_{QL^2(\Omega)}.
 \intertext{If we apply Young's inequality~\eqref{eqn:youngs-ineq} m with $p=2$ and  $\eta= \epsilon^{-\frac{(\mu+1)}{2}}K^{\frac{\mu+1}{2}}$, we get }
 & \leq C_0S(2,A)\epsilon^\frac{1}{2}\|\nabla u \|_{QL^2(\Omega)}^2
 + C_0\epsilon^\frac{1}{2}\|\nabla u \|_{QL^2(\Omega)}^2+ C_0K^{\mu+1}\epsilon^{-\frac{2\mu+1}{2}}\|u\|_{L^2(v,\Omega)}^2.
\end{align*}
Therefore, if we let $\epsilon=(4C_0S(2,A)+C_0)^{-2}$ and $C_1= C_0K^{\mu+1}\epsilon^{-\frac{2\mu+1}{2}}$, we have shown that
   \begin{equation}\label{eq:I1 Orlicz}
    I_1\leq \frac{1}{4}\|\nabla u\|_{QL^2(\Omega)}^2+C_{1}\|u\|_{L^2(v,\Omega)}^2.
\end{equation}

\medskip

We now estimate $I_2$ in \eqref{eq-coer}.  We proceed as above:  we apply the Orlicz H\"older inequality,  the interpolation inequality~\eqref{Orlicz interpolation}, and the Sobolev inequality to get
\begin{align*} 
    I_2
    & \leq 2\|F\|_{L^D(v,\Omega)}\|u^2\|_{L^{\Bar{D}}(v,\Omega)} \\
    & \leq 2\|F\|_{L^D(v,\Omega)}(\epsilon\|u\|_{L^A(v,\Omega)}^2+K^{\mu+1}
        \epsilon^{-\mu}\|u\|_{L^2(v,\Omega)}^2) \\
    & \leq 2\|F\|_{L^D(v,\Omega)}(\epsilon S(2,A)^2 \|\nabla u\|_{QL^2(\Omega)}^2 +K^{\mu+1}
        \epsilon^{-\mu}\|u\|_{L^2(v,\Omega)}^2). 
\end{align*}
If we now let $\epsilon=(8{S(2,A)}^2\|F\|_{L^\Phi(v,\Omega)})^{-1}$, we get
\begin{equation}\label{eq:I2 Orlicz}
    I_2\leq \frac{1}{4}\|\nabla u\|_{QL^2(\Omega)}^2+C_{2}\|u\|_{L^2(v,\Omega)}^2.
\end{equation}
where $C_{2}=2K^{\mu+1}\epsilon^{-\mu}\|F\|_{L^\Phi(v,\Omega)}$.

If we now combine \eqref{eq:I1 Orlicz} and \eqref{eq:I2 Orlicz} with \eqref{eq-coer}, and use the Sobolev inequality with no gain, we get
\begin{align*} \mathcal{L}(u,u)
& \geq \|\nabla u\|_{QL^2(\Omega)}^2-\frac{1}{2}\|\nabla u\|_{QL^2(\Omega)}^2-(C_{1}+C_2)\| u\|_{L^2(v,\Omega)}^2 \\
& \geq c_1 \|u\|_{QH^{1,2}(v,\Omega)} - C_3\| u\|_{L^2(v,\Omega)}^2,
\end{align*}
where $C_3=C_{1}+C_{2}$ and $c_1=(\sqrt{2}(1+S(2,1)))^{-2}$. This completes the proof.
\end{proof}

\medskip

We now consider the case of a Sobolev inequality with no gain.  

\begin{lemma}\label{no-gain-boundedness} 
Given  Hypothesis~\ref{hyp:global-hyp}, suppose that the global degenerate Sobolev inequality~\eqref{eqn:global-sobolev} holds without gain (i.e., $\sigma=1$). 
If $F\in L^{\infty}(v,\Omega)$ and $\bH,\,\bG \in L^{\infty}(v,\Omega,\R^N)$,  then the bilinear form $\mathcal{L}$ is bounded in $QH_0^{1,2}(v,\Omega)$: that is, there is a $C>0$  such that for all $u,\,w \in QH_0^{1,2}(v,\Omega)$,
    \begin{equation*}
        |\mathcal{L}(u,w)|\leq C\|u\|_{QH^{1,2}(v,\Omega)}\cdot\|w\|_{QH^{1,2}(v,\Omega)}.
    \end{equation*}
\end{lemma}

\begin{proof}
Let $u,w \in QH_0^{1,2}(v,\Omega)$. We argue as in~\eqref{first part of bil.bdd} but instead of using H\"older's inequality we use the fact that all the coefficients are in $L^\infty$ to get
\begin{align*}
   |\mathcal{L}(u,w)|
   &\leq  \|u\|_{QH^{1,2}(v,\Omega)}\|w\|_{QH^{1,2}(v,\Omega)}  \\
   & \qquad \quad + C(\bR)\|\bH\|_{L^{\infty}(v,\Omega)}\|w\|_{L^2(v,\Omega)}\|u\|_{QH^{1,2}(v,\Omega)}\\
    &\qquad \quad+ C(\bS)\|\bG\|_{L^{\infty}(v,\Omega)}\|u\|_{L^2(v,\Omega)}\|w\|_{QH^{1,2}(v,\Omega)} \\
    & \qquad \quad  +\|F\|_{L^{\infty}(v,\Omega)}\|u\|_{L^2(v,\Omega)}\|w\|_{L^2(v,\Omega)}. 
 \intertext{If we now apply the Sobolev inequality, we get}
  & \leq  \|u\|_{QH^{1,2}(v,\Omega)}\|w\|_{QH^{1,2}(v,\Omega)}  \\
   & \qquad \quad + C(\bR)S(2,1)\|\bH\|_{L^{\infty}(v,\Omega)}
   \|\nabla w\|_{QL^2(\Omega)}\|u\|_{QH^{1,2}(v,\Omega)}\\
    &\qquad \quad+ C(\bS)S(2,1)\|\bG\|_{L^{\infty}(v,\Omega)}
    \|\nabla u\|_{QL^2(\Omega)}\|w\|_{QH^{1,2}(v,\Omega)} \\
    & \qquad \quad  +S(2,1)^2\|F\|_{L^{\infty}(v,\Omega)}\|\nabla u\|_{QL^2(\Omega)}\|\nabla w\|_{QL^2(\Omega)} \\ 
    & \leq  C\|u\|_{QH^{1,2}(v,\Omega)}\|w\|_{QH^{1,2}(v,\Omega)},
\end{align*}
where 
$$C=1+ C(\bR)S(2,1)\|\bH\|_{L^{\infty}(v,\Omega)}+ C(\bS)S(2,1)\|\bG\|_{L^{\infty}(v,\Omega)}+S(2,1)^2\|F\|_{L^\infty(v,\Omega)}.$$
\end{proof}

\begin{lemma}\label{no-gain-almost-coercivity}
Given  Hypothesis~\ref{hyp:global-hyp}, suppose that the global degenerate Sobolev inequality~\eqref{eqn:global-sobolev} holds without gain (i.e., $\sigma=1$). 
If $F\in L^{\infty}(v,\Omega)$ and $\bH,\,\bG \in L^{\infty}(v,\Omega,\R^N)$,  then the bilinear form $\mathcal{L}$ is almost coercive:  that is, there exist constants $c_1,\,C_3>0$ such that for every $u \in QH^{1,2}_0(v,\Omega)$,
\begin{equation*}
     \mathcal{L}(u,u) \geq c_1\|u\|^2_{QH^{1,2}(v,\Omega)} - C_3\|u\|^2_{L^2(v,\Omega)}.
    \end{equation*}
\end{lemma}

\begin{proof}
   We argue as we did above, starting with \eqref{eq-coer} and \eqref{eq:I1}.  Instead of applying H\"older's inequality we use the fact that the coefficients are in $L^\infty$ and Young's inequality~\eqref{eqn:youngs-ineq} with $p=2$ and $\eta=\epsilon$ (to be chosen below) to get
\begin{align*}
I_1 
& \leq S(2,1)(C(\bR)\|\bH\|_{L^\infty(\Omega)} + C(\bS)\|\bG\|_{L^\infty(\Omega)})
\|u\|_{L^2(v,\Omega)}\|\nabla u\|_{QL^2(\Omega)} \\
& \leq S(2,1)(C(\bR)\|\bH\|_{L^\infty(\Omega)} + C(\bS)\|\bG\|_{L^\infty(\Omega)})
\big(\epsilon\|\nabla u\|^2_{QL^2( \Omega)} + \epsilon^{-1}\| u\|^2_{L^2(v,\Omega)}\big).
\end{align*}     
If we let $\epsilon=\big( 2S(2,1)(C(\bR)\|\bH\|_{L^\infty(\Omega)} + C(\bS)\|\bG\|_{L^\infty(\Omega)})\big)^{-1}$, we get
   \begin{equation}\label{eq:I1 Infty}
    I_1\leq \frac{1}{2}\|\nabla u\|_{QL^2(\Omega)}^2+C_1\|u\|_{L^2(v,\Omega)}^2
\end{equation}
where $C_1= \epsilon^{-1}S(2,1)(C(\bR)\|\bH\|_{L^\infty(\Omega)} + C(\bS)\|\bG\|_{L^\infty(\Omega)})$.

\medskip

The estimate for $I_2$ is immediate:

\begin{equation} \label{eqn:I2-infty}
      I_2\leq \int_\Omega |F||u^2|\,vdx\leq \|F\|_{L^\infty(\Omega)}\|u\|_{L^{2}(v,\Omega)}^2 = C_2\|u\|_{L^{2}(v,\Omega)}^2.
  \end{equation}
  Therefore, if we insert \eqref{eq:I1 Infty} and \eqref{eqn:I2-infty} into \eqref{eq-coer}, and apply the Sobolev inequality, we get
  \begin{equation*}
        \mathcal{L}(u,u)
         \geq \frac{1}{2}\|\nabla u\|_{QL^2(\Omega)}^2-C_3\| u\|_{L^2(v,\Omega)}^2 
         \geq c_1\|u\|_{QH^{1,2}(v,\Omega)}-C_3\| u\|_{L^2(v,\Omega)}^2, 
\end{equation*}
where $c_1=\big(\sqrt{2}(S(2,1)+1)\big)^{-2}$ and  $C_3=C_1+C_2$. 
\end{proof}

\section{Proof of main results}
\label{section:main-proofs}

In this section we prove Theorems~\ref{thm:main-theorem}--\ref{thm:no-gain}. For clarity, in this section we will not   follow the convention used above of writing (for instance) $u\in QH^{1,2}(v,\Omega)$.  Here we will use the fact that $QH^{1,2}_0(v,\Omega)$ is a Hilbert space whose elements are pairs: ${\bf u} = (u,\nabla u)\in L^2(v,\Omega)\oplus QL^2(\Omega)$.  Define the projection mapping $\mathcal{P}_2:QH^{1,2}_0(v,\Omega)\rightarrow L^2(v,\Omega)$  by $\cP_2({\bf u}) = u$.  In our proof below we will need that $\mathcal{P}$ is a compact operator.  We prove this in the next two lemmas.  Since without additional work we can prove a more general result, we will prove that for $1< p<\infty$, the  projection  $\mathcal{P}_p:QH_0^{1,p}(v,\Omega)\rightarrow L^p(v,\Omega)$, $\cP_p({\bf u}) = u$, is compact.

\begin{lemma} \label{lemma:compact-gain}
Given $1<p<\infty$ and $\Omega$  as in Hypothesis~\ref{hyp:global-hyp}, suppose $v\in L^1(\Omega)$.  Also suppose that either the global degenerate Sobolev inequality~\eqref{eqn:global-sobolev} holds with gain $\sigma>1$, or the  Sobolev inequality~\eqref{eqn:global-sobolev-orlicz} holds with $A(t)=t^p\log(e+t)^\sigma$, $\sigma>0$.  Suppose further that the pair $(\Omega, \rho)$  is a quasimetric space that satisfies Hypothesis~\ref{hyp:poincare} and the local  degenerate Poincar\'e inequality~\eqref{eqn:local-poincare-QM}  holds on $\Omega$.  Then the projection mapping $\cP_p:QH_0^{1,p}(v,\Omega)\rightarrow L^p(v,\Omega)$ is compact.
\end{lemma}

The first part of Lemma~\ref{lemma:compact-gain}, assuming a Sobolev inequality with power gain, is a special case of a result due to Monticelli and the fourth author~\cite[Theorem~1.8]{MR4069608}.  The proof of this result was only sketched.  Since we can adapt the short proof of the case we need to also prove the second part, we include the details.

\begin{remark}
    We will only use Lemma~\ref{lemma:compact-gain}  when we are dealing with lower order terms, so we will assume that $v\in L^1(\Omega)$ and $|Q(x)|_\op \leq kv(x)$ a.e.  However, we do not need the second hypothesis to prove this lemma.
\end{remark}

\begin{proof}
    We first assume that inequality~\eqref{eqn:global-sobolev} holds.  To show that $\cP_p$ is compact, fix any sequence $\{ \vecu_n \}_{n=1}^\infty  = \{(u_n,\nabla u_n)\}_{n=1}^\infty$ in $QH_0^{1,p}(v,\Omega)$ such that for all $n$, $\|\vecu_n \|_{QH^{1,p}(v,\Omega)} \leq M$.  We will show that there exists a subsequence that is Cauchy in $L^p(v,\Omega)$ and so converges.   
    
    Since $\{u_n\}_{n=1}^\infty$ is uniformly bounded in $L^p(v,\Omega)$, which is reflexive, it has a  weakly convergent subsequence (see Yosida~\cite[p.~126]{MR1336382}).  Therefore, by passing to a subsequence, we may assume that $\{u_n\}_{n=1}^\infty$ is weakly Cauchy in $L^p(v,\Omega)$.   We will show that it is strongly Cauchy.  
    
    Fix $\epsilon>0$.  Since $v\in L^1(\Omega)$, there exists a compact set $K\subset \Omega$ such that
    \[ v(\Omega\setminus K)^{\frac{1}{p\sigma'}} \leq \epsilon. \]
    Therefore, for all $n>m$, 
    \[ 
        \|u_n-u_m\|_{L^p(v,\Omega)} 
    \leq 
    \|u_n-u_m\|_{L^p(v,K)}
    +
    \|u_n-u_m\|_{L^p(v,\Omega\setminus K)} = I_1 + I_2. \]
We will estimate each term separately.  The second is straightforward:  by H\"older's inequality and the Sobolev inequality~\eqref{eqn:global-sobolev},
\begin{multline*} I_2 
\leq \|u_n-u_m\|_{L^{ p\sigma}(v,\Omega)} v(\Omega\setminus K)^{\frac{1}{p\sigma'}} \\
\leq S(p,\sigma)\|\nabla u_n-\nabla u_m\|_{QL^p(\Omega)}v(\Omega\setminus K)^{\frac{1}{p\sigma'}}
\leq  2MS(p,\sigma)v(\Omega\setminus K)^{\frac{1}{p\sigma'}}
    \leq  2MS(p,\sigma)\epsilon. 
\end{multline*}

To estimate $I_1$, note that since the $\rho$-balls satisfy the geometric doubling property, by~\cite[Lemma~2.3]{MR4069608} there exists $\delta_0>0$ such that for every $0<r<\delta_0$, there is a finite collection of $\rho$-balls $\{B_j\}_{j=1}^J = \{ B_\rho(x_j,r)\}_{j=1}^J$ with $x_j \in K$ such that
\[  K \subset \bigcup_{j=1}^J B_j \subset \bigcup_{j=1}^J cB_j \subset \Omega, \]
where $c\geq 1$ is the constant in the Poincar\'e inequality~\eqref{eqn:local-poincare-QM}, and 
\[  \sum_{j=1}^J \chi_{cB_j}(x) \leq P. \]
The constant $P$ is independent of the choice of $r$.  Fix such a collection of balls with $r$ such that the constant in~\eqref{eqn:local-poincare-QM} is $\epsilon$.  Then we have that 
\begin{align*} 
I_1^p 
& \leq \sum_{j=1}^J \int_{B_j} |u_n-u_m|^p\,vdx \\
& \leq 2^p \sum_{j=1}^J \int_{B_j} |u_n-u_m- (u_n-u_m)_{B_j,v}|^p\,vdx  
+ 2^p \sum_{j=1}^J |(u_n-u_m)_{B_j,v}|^p v(B_j) \\
& = 2^pI_3 + 2^p I_4.
\end{align*}

We again estimate each term separately.  To estimate $I_3$ we use the Poincar\'e inequality:
\begin{align*}
 I_3 
& \leq \epsilon^p \sum_{j=1}^J \|\vecu_n - \vecu_m\|_{QH^{1,p}(v,cB_j)}^p \\
& \leq \epsilon^p 2^p\sum_{j=1}^J \big[ \|u_n-u_m\|_{L^p(v,cB_j)}^p + \|\nabla u_n-\nabla u_m\|_{QL^p(cB_j)}^p] \\
& \leq \epsilon^p 2^pP \|\vecu_n - \vecu_m\|_{QH^{1,p}(v,\Omega)}^p \\
& \leq \epsilon^p 2^{2p}PM^p.
\end{align*}

Finally, to estimate $I_4$, since the sequence $\{u_n\}_{n=1}^\infty$ is weakly Cauchy in $L^p(v,\Omega)$, and since $\chi_{B_j} \in L^{p'}(v,\Omega)$, we have that there exists $L>1$ such that if $n>m\geq L$, then
\[ I_4 = \sum_{j=1}^J v(B_j)^{1-p}\bigg|\int_{\Omega} (u_n-u_m)\chi_{B_j}\,vdx\bigg|^p \leq \epsilon^p.  \]

If we combine the above estimates we see that for all such $n,\,m$ we have that
\[ \|u_n-u_m\|_{L^p(v,\Omega)} \leq 2^3PM\epsilon + 2\epsilon + 2MS(p,\sigma)\epsilon. \]
It follows that the sequence $\{u_n\}_{n=1}^\infty$ is strongly Cauchy.  

\medskip

Now suppose that~\eqref{eqn:global-sobolev-orlicz} holds with $A(t)=t^p\log(e+t)^\sigma$, $\sigma>0$. The only part of the above proof that changes is the estimate for $I_2$.  Fix $K\subset \Omega$ such that 
\[ \|\chi_{\Omega\setminus K}\|_{L^{\bar{A}}(v,\Omega)} = \bar{A}^{-1}\big(v(\Omega\setminus K)^{-1}\big)^{-1} 
\approx v(\Omega\setminus K)^{\frac{1}{p'}}\log(e+v(\Omega\setminus K)^{-1})^{-\frac{\sigma}{p}} \leq \epsilon.\]
Then by H\"older's inequality in the scale of Orlicz spaces~\eqref{eqn:orlicz-holder} and~\eqref{eqn:global-sobolev-orlicz},
\begin{multline*}
I_2 \leq 2\|u_n-u_m\|_{L^A(v,\Omega)}\|\chi_{\Omega\setminus K}\|_{L^{\bar{A}}(v,\Omega)} \\
\leq 2S(p,A)\|\nabla u_n-\nabla u_m\|_{QL^p(\Omega)}\|\chi_{\Omega\setminus K}\|_{L^{\bar{A}}(v,\Omega)}
\leq 4MS(p,A)\epsilon.
\end{multline*}
\end{proof}

\begin{lemma} \label{lemma:compact-no-gain}
 Given $1<p<\infty$ and $\Omega$  as in Hypothesis~\ref{hyp:global-hyp}, and a bounded open set $\Theta$ such that $\bar{\Omega}\subset \Theta$,  suppose the pair $(\Theta,\rho)$ is a quasimetric space  that satisfies Hypothesis~\ref{hyp:poincare} and the local  Poincar\'e inequality~\eqref{eqn:local-poincare-QM}  holds on $\Theta$. Then the projection mapping $\cP_p:QH_0^{1,p}(v,\Omega)\rightarrow L^p(v,\Omega)$ is compact.
\end{lemma}

A more general version of Lemma~\ref{lemma:compact-no-gain} was proved in~\cite[Theorem~1.7]{MR4069608}.  The proof of our result is obtained  by modifying the proof of Lemma~\ref{lemma:compact-gain} and so for completeness we sketch the changes.

\begin{proof}
    The proof is essentially the same as the proof of Lemma~\ref{lemma:compact-gain}.  We replace $\Omega$ by $\Theta$, and $K$ by $\bar{\Omega}$.  To prove that the sequence $\{u_n\}_{L^p(v,\Omega)}$ is weakly Cauchy we do not need to estimate the integral on $\Theta\setminus \bar{\Omega}$, so we do not need to estimate $I_2$, where we used the Sobolev inequality.  The rest of the proof is the same.
\end{proof}

\begin{remark}
In the proof of Lemma~\ref{lemma:compact-gain}, we used the hypothesis $v\in L^1(\Omega)$, along with the Sobolev inequality, to estimate $I_2$.  Since we do not need to estimate $I_2$ in the proof of Lemma~\ref{lemma:compact-no-gain}, we also do not need an integrability assumption on $v$.
\end{remark}

\medskip
\begin{proof}[Proof of Theorems~\ref{thm:main-theorem},~\ref{thm:log-gain}, and~\ref{thm:no-gain}]
The proofs of these results are essentially the same.  We will use the Fredholm alternative to show that a unique solution exists; the results we proved above will show that the hypotheses of this result are satisfied.

Given an element ${\bf u}=(u,\nabla u) \in QH^{1,2}_0(\Omega)$, it is a degenerate weak solution of the Dirichlet problem \eqref{eqn:dirichlet} if and only if
\begin{equation}\label{scottsol1}
\mathcal{L}({\bf u},{\bf h}) = \mathcal{F}({\bf h})
\end{equation}
for all ${\bf h}=(h,\nabla h)\in QH^{1,2}_0(v,\Omega)$, where $\mathcal L$ is the bilinear form~\eqref{bilinear} and the linear functional $\mathcal{F}$ in  $QH^{1,2}_0(v,\Omega)^*$ is given by
$$\mathcal{F}({\bf h})= \int_\Omega fh\,vdx + \int_\Omega {\bf g}\cdot \bT{ h}\,vdx.$$
As we proved above (Lemmas~\ref{lemma:bounded-power-gain} and~\ref{power-gain-almost-coercivity} for Theorem~\ref{thm:main-theorem}, Lemmas~\ref{orlicz-gain-boundedness} and~\ref{orlicz-gain-almost-coercivity} for Theorem~\ref{thm:log-gain}, and Lemmas~\ref{no-gain-boundedness} and~\ref{no-gain-almost-coercivity} for Theorem~\ref{thm:no-gain}) the bilinear form $\mathcal{L}$ is bounded and almost coercive on $QH^{1,2}_0(v,\Omega)$.  Therefore, we can define a bounded and coercive operator as follows.  Define  $\mathcal{T}:L^2(v,\Omega)\rightarrow L^2(v,\Omega)^*$ by 
$$\mathcal{T}u(w) = \int_\Omega uw\,vdx$$
for all $w\in L^2(v,\Omega)$.  Then, for $\beta>C_3$ (as in Lemmas~\ref{power-gain-almost-coercivity},~\ref{orlicz-gain-almost-coercivity}, and~\ref{no-gain-almost-coercivity}), the associated bilinear form given by
$$\mathcal{B}({\bf u},{\bf h}) = \mathcal{L}({\bf u},{\bf h})  + \beta\int_\Omega uh\,vdx$$
is both bounded and coercive.  Hence, by the Lax-Milgram Theorem~\cite[Theorem~5.8]{MR1814364}, for every $\mathcal{G}\in QH^{1,2}_0(v,\Omega)^*$, there is a unique ${\bf w}\in QH^{1,2}_0(v,\Omega)$ such that
\begin{align}\label{augmentedBilinear}
\mathcal{B}({\bf w},{\bf h}) = \mathcal{G}({\bf h})
\end{align}
for all ${\bf h}\in QH^{1,2}_0(v,\Omega)$.  Thus, we can define a linear, invertible ``solution map" $J:QH^{1,2}_0(v,\Omega)^*\rightarrow QH^{1,2}_0(v,\Omega)$ by 
$${\bf w} = J(\mathcal{G})$$
where ${\bf w}$ is the unique element of $QH^{1,2}_0(v,\Omega)$ satisfying \eqref{augmentedBilinear}.

If we now add $\beta \mathcal{T}\mathcal{P}_2{\bf u}$ to each side, we see that equation \eqref{scottsol1} is equivalent to the equation 
\begin{align}\label{scottsol2}
\mathcal{B}({\bf u},{\bf h}) = \mathcal{G}({\bf h})
\end{align}
for all ${\bf h}\in QH^{1,2}_0(v,\Omega)$ where the linear functional $\mathcal{G}\in QH^{1,2}_0(v,\Omega)^*$ is defined by 
$$\mathcal{G} = \mathcal{F} + \beta\mathcal{T}\mathcal{P}_2({\bf u}).$$
Furthermore, ${\bf u}$ satisfies \eqref{scottsol2} if and only if $J\mathcal{B}({\bf u},\cdot)={\bf u}$.  Thus, $\vecu$ is a weak solution of our Dirichlet problem if and only if ${\bf u}$ satisfies the equation
\begin{align}\label{scottsol3}
(I-K){\bf u} = {\bf s}
\end{align}
where $K:QH^{1,2}_0(v,\Omega)\rightarrow QH^{1,2}_0(v,\Omega)$ is defined by $K=\beta J\mathcal{T}\mathcal{P}_2$ and ${\bf s} = J\mathcal{F}\in QH^{1,2}_0(\Omega)$.  Since each of the mappings $J,\mathcal{G},\mathcal{T}$ is continuous, $K$ is a compact operator if the projection map $\mathcal{P}_2$ is compact.  But this follows from Lemma~\ref{lemma:compact-gain} (for Theorems~\ref{thm:main-theorem} and~\ref{thm:log-gain}) and Lemma~\ref{lemma:compact-no-gain} (for Theorem~\ref{thm:no-gain}).   Therefore, by the Fredholm alternative \cite[Theorem~5.11]{MR1814364}, for each ${\bf s}\in QH^{1,2}_0(\Omega)$, equation \eqref{scottsol3} has exactly one solution ${\bf u} \in QH^{1,2}_0(v,\Omega)$ precisely when the kernel of $I-K$ is trivial.  Since ${\bf w}\in \textrm{Ker}(I-K)$ if and only if ${\bf w}=K{\bf w}$, if we apply $J^{-1}$ to each side we get that for every ${\bf h}\in QH^{1,2}_0(v,\Omega)$,
$$\mathcal{B}({\bf w},{\bf h}) = \beta\int_\Omega wh\,vdx.$$
Equivalently,  ${\bf w}\in \textrm{Ker}(I-K)$ if and only if
$$\mathcal{L}({\bf w},{\bf h}) = 0$$
for ${\bf h}\in QH^{1,2}_0(v,\Omega)$.
That is, ${\bf w}$ is a weak solution of the zero problem~\eqref{homogeneous problem}.
By Theorem~\ref{thm:hom-problem} (which holds given the hypotheses of Theorems~\ref{thm:main-theorem},~\ref{thm:log-gain}, or~\ref{thm:no-gain}),  $(0,0)\in QH^{1,2}_0(v,\Omega)$ is the unique degenerate weak solution of the zero problem, and so $\textrm{Ker}(I-K)$ is  trivial.  Therefore, we have shown that for every ${\bf s}\in QH^{1,2}_0(v,\Omega)$, there is a unique solution ${\bf u}\in QH^{1,2}_0(v,\Omega)$ of \eqref{scottsol3} and so there is a unique degenerate weak solution to the Dirichlet problem \eqref{eqn:dirichlet}.
\end{proof}

\medskip

\begin{proof}[Proof of Theorem~\ref{thm:no-gain-no-first-order}]
    The proof of this result is much simpler than the others.  Just assuming $v$ is nonnegative and measurable (i.e., a weight), we have that $QH_0^{1,2}(v,\Omega)$ is still a Hilbert space.  
    Since $\bG=\bH=0$, then it follows from the proofs of Lemmas~\ref{no-gain-boundedness} and~\ref{no-gain-almost-coercivity} that the bilinear form $\mathcal L$ is bounded and coercive.  (We use the Sobolev inequality without gain and the fact that $F$ is nonnegative, instead of the compatibility condition~\eqref{eqn:compatible}, to prove coercivity.)  Therefore, with the notation used above, by the Lax-Milgram theorem there exists a unique $\vecu$ such that
    \[ {\mathcal L}(\vecu,\vech) = {\mathcal F}(\vech); \]
    equivalently, $\vecu$ is the unique degenerate weak solution to the Dirichlet problem \eqref{eqn:dirichlet}.
\end{proof}

\section{Proof of Corollary~\ref{cor:boundary-data}}
\label{section:cor-proofs}

Fix $\varphi\in QH^{1,2}(v,\Omega)$.  Recall that an element $u\in QH^{1,2}(v,\Omega)$ is a degenerate weak solution of the Dirichlet problem with boundary data $\varphi$,
\begin{equation}\label{inhomogeneousxx}
\begin{cases}
Lu = f+T'{\bf g}, & x\in \Omega,\\
u =\varphi, & x\in \partial\Omega,
\end{cases}
\end{equation}
if $w=u-\varphi \in QH^{1,2}(v,\Omega)$ is a solution of the Dirichlet problem 
\begin{equation}\label{inhomogeneous2xx}
\begin{cases}
Lw = f+{\bf T}'{\bf g} - L\varphi, & x \in \Omega,\\
w=0, & x\in \partial\Omega.
\end{cases}
\end{equation}
That is, for each $y \in QH^{1,2}_0(v,\Omega)$,
\begin{align*}
& \int_\Omega \nabla w \cdot Q\nabla y\,dx + \int_\Omega \bH\cdot\bR w y\,vdx 
+ \int_\Omega w\bG \cdot \bS y\,vdx + \int_\Omega Fwy\,vdx\\
&\qquad = \int_\Omega fy \,vdx + \int_\Omega {\bf g}\cdot {\bf T}y \,vdx \\
&\qquad \qquad - \int_\Omega \nabla \varphi\cdot Q\nabla y~dx -\int_\Omega \varphi \bH\cdot \bR y \,vdx 
- \int_\Omega \varphi \bG\cdot \bS y\,vdx - \int_\Omega Fy \varphi \,vdx\\
& \qquad =\int_\Omega \left[f - \varphi \bH\cdot \bR -F\varphi \right]y \,vdx \\
& \qquad \qquad +\int_\Omega \left[ {\bf g}\cdot {\bf T} - \frac{1}{\sqrt{v}}\sqrt{Q}\nabla \varphi \cdot \frac{1}{\sqrt{v}}\sqrt{Q}\nabla -\varphi\bG \cdot \bS \right]y\,vdx\\
&= \int_\Omega f_1y\,vdx + \int_\Omega {\bf g}_1{\cdot \bf T}_{1}y \,vdx, 
\end{align*}

where 
$f_1 = f - {\bf HR}\varphi - F\varphi$, 
$${\bf g}_1 = (g_1,\ldots,g_N,-V_1\varphi,\ldots,-V_n\varphi,-\varphi G_1,\ldots,-\varphi G_N), $$
and
$${\bf T_1} = (T_1,\ldots,T_N,V_1,\ldots,V_n,S_1,\ldots,S_N),$$
with $V_i = \frac{1}{\sqrt{v}}{\bf q_i}\cdot\nabla$ and ${\bf q_i}$ the $i$-th row vector of $\sqrt{Q}$.

It follows from Hypothesis~\ref{hyp:global-hyp} that  $f_1\in L^2(v,\Omega)$, ${\bf g_1}\in L^2(v,\Omega,\R^{2N+n})$, and  ${\bf T_1}$ a $(2N+n)$-tuple of degenerate subunit vector fields that satisfy \eqref{eqn:norm-subunit}.  Thus, if the hypotheses of one of Theorems \ref{thm:main-theorem}, \ref{thm:log-gain}, \ref{thm:no-gain}, or \ref{thm:no-gain-no-first-order}  hold, then this implies that there exists a unique degenerate weak solution $w\in QH^{1,2}_0(v,\Omega)$ of the Dirichlet problem \eqref{inhomogeneous2xx}.  Hence, $u=w+\varphi$ is a degenerate weak solution of \eqref{inhomogeneousxx}.  This completes the proof.


\bibliographystyle{plain}
\bibliography{tubitak-bibliography}

\end{document}